\documentclass[12pt]{amsart}
\usepackage{a4wide,mathdots,comment,amssymb,mathscinet,amsrefs}
\usepackage[hyperfootnotes=false]{hyperref}
\usepackage[stable]{footmisc}

\newcommand{\GL}{\operatorname{GL}}
\newcommand{\Sp}{\operatorname{Sp}}
\newcommand{\Speh}{\mathfrak{S}}
\newcommand{\GSp}{\operatorname{GSp}}
\newcommand{\bs}{\backslash}
\newcommand{\Irr}{\operatorname{Irr}}
\newcommand{\disc}{\operatorname{disc}}
\newcommand{\pl}{\operatorname{pl}}
\newcommand{\Bes}{\mathcal{B}}
\newcommand{\OOO}{\mathcal{O}}
\newcommand{\basis}{\mathfrak{B}}
\newcommand{\ind}{\operatorname{ind}}
\newcommand{\Ind}{\operatorname{Ind}}
\newcommand{\embd}{\iota}
\newcommand{\K}{\mathbf{K}}
\newcommand{\tr}{\operatorname{tr}}
\newcommand{\Z}{\mathbb{Z}}
\newcommand{\smth}{\operatorname{sm}}
\newcommand{\mira}{Q}
\newcommand{\rest}{\big |}
\newcommand{\dirac}{\pmb{\delta}}
\newcommand{\modulus}{\delta}
\newcommand{\trns}{\mathbf{T}}
\newcommand{\temp}{\operatorname{temp}}
\newcommand{\Mat}{\operatorname{Mat}}
\newcommand{\abs}[1]{\left|{#1}\right|}
\newcommand{\Conv}{\mathfrak{C}}
\newcommand{\Nu}{\mathcal{V}}

\newcommand{\swrz}{\mathcal{S}}
\newcommand{\bil}{\mathbf{B}}
\newcommand{\C}{\mathbb{C}}
\newcommand{\data}{\mathcal{D}}
\newcommand{\crnr}{\iota_1}
\newcommand{\id}{\operatorname{Id}}
\newcommand{\End}{\operatorname{End}}
\newcommand{\Hom}{\operatorname{Hom}}
\newcommand{\model}{\mathfrak{M}}
\newcommand{\trans}{\mathcal{T}}
\newcommand{\transs}{\trans^*}
\newcommand{\triv}{{\bf{1}}}
\newcommand{\sm}[4]{\left(\begin{smallmatrix}{#1}&{#2}\\{#3}&{#4}\end{smallmatrix}\right)}
\newcommand{\proj}{\mathfrak{p}}
\newcommand{\der}{{\operatorname{der}}}
\newcommand{\reg}{{\operatorname{reg}}}

\newcommand{\M}{\mathcal{M}}
\newcommand{\Id}{{\operatorname{Id}}}
\newcommand{\B}{}

\newcommand{\reltempw}{\mathcal{C}^w}
\newcommand{\reltemp}{\mathcal{C}_0^w}

\newtheorem*{theorem}{Theorem}
\newtheorem*{lemma}{Lemma}

\newtheorem*{proposition}{Proposition}
\newtheorem*{corollary}{Corollary}
\newtheorem*{example}{Example}
\newtheorem*{definition}{Definition}
\newtheorem*{remark}{Remark}

\numberwithin{equation}{section}

\begin{document}

\title{Explicit Plancherel formula for the space of symplectic forms}
\author{Erez Lapid}
\address{Department of Mathematics, Weizmann Institute of Science, Rehovot 7610001, Israel}
\email{erez.m.lapid@gmail.com}
\author{Omer Offen}
\address{Department of Mathematics, Brandeis University, 415 South Street, Waltham, MA 02453, USA}
\email{offen@brandeis.edu}

\date{\today}

\begin{abstract}
We provide a Plancherel decomposition for the space of symplectic bilinear forms of rank $2n$
over a local non-archimedean field $F$ in terms of that of $\GL_n(F)$.
\end{abstract}

\maketitle

\setcounter{tocdepth}{1}
\tableofcontents

\section{Introduction}

\subsection{}
The purpose of this paper is to give an explicit Plancherel formula for
the space $Y_n$ of symplectic (i.e., non-degenerate alternating) bilinear forms on a $2n$-dimensional
vector space over a local field $F$, in terms of the Plancherel formula for $G'=\GL_n(F)$.
For simplicity we will assume that $F$ is non-archimedean (of any characteristic)
although the statement and the idea of the proof should hold for the archimedean case as well.

The Plancherel decomposition for general real reductive symmetric spaces was worked out some time ago by many mathematicians
and is one of the highlights of harmonic analysis in the post Harish-Chandra era.
(See \cite{MR1957064} and the references therein.)
In the non-archimedean case, one cannot expect a completely explicit Plancherel decomposition,
as even in the group case, there is no explicit description of the discrete series
(or their characters for that matter).\footnote{By the local Langlands conjecture,
which is a theorem in many cases, discrete series correspond to certain representations
of the Weil group of $F$ with some additional data.}
Several years ago, Sakellaridis--Venkatesh made remarkable conjectures on the $L^2$-decomposition of $p$-adic symmetric spaces
(and in fact, more generally, of spherical varieties) \cite{MR3764130}.
In particular, the support of the Plancherel measure is expected to be the image of functoriality
from a certain group prescribed by the spherical variety.
Sakellaridis--Venkatesh also expressed, at least under certain assumptions, the continuous part of the spectrum in terms of
the discrete spectrum of smaller spherical varieties.
(For symmetric spaces this work was completed by Delorme \cite{MR3770165}.)

A different approach to study the spectrum of symmetric spaces was recently taken by Beuzart-Plessis \cite{1812.00047}.
We will follow the latter.

An explicit spherical Plancherel formula for $Y_n/\GL_{2n}(\OOO)$, where $\OOO$ is the ring of integers of $F$, was obtained in \cite{MR944325}.
\subsection{}
In order to state our main result we first set some notation.
Fix a local non-archimedean field $F$ with normalized absolute value $\abs{\cdot}$ and ring of integers $\OOO$.

If $X$ is an $\ell$-space \cite{MR0425030}, we denote by $C^\infty(X)$ the space of locally constant, complex-valued functions on $X$
and by $\swrz(X)$ the subspace of compactly supported functions in $C^\infty(X)$.
If $G$ is an $\ell$-group, $H$ is a closed subgroup of $G$ and $dh$ is a Haar measure on $H$, then
the projection $f\mapsto\int_Hf(hg)\ dh$ identifies $\swrz(H\bs G)$ with the space of left $H$-coinvariants of $\swrz(G)$.

From now on, fix an integer $n\ge1$ and let $G=\GL_{2n}(F)$, acting on the right on the vector space $F^{2n}$ of row vectors
of size $2n$, with its standard basis $e_1,\dots,e_{2n}$. Let $H\subset G$ be the symplectic group
\[
H=\Sp_n(F)=\{g\in G:\,^tgJ_ng=J_n\}
\]
with respect to $J_n=\sm{}{w_n}{-w_n}{}$ where $w_n=\left(\begin{smallmatrix}&&1\\&\iddots\\1&&\end{smallmatrix}\right)\in G'$.
Thus, $Y_n\simeq H\bs G$.

For any tempered, irreducible representation $\pi$ of $G'$
let $\sigma=\Speh(\pi)$ be the corresponding Speh representation of $G$.
(All representations are over the complex numbers.)
More precisely, let $P=M\ltimes U$ be the standard parabolic subgroup of $G$ of type $(n,n)$ with its standard
Levi decomposition, i.e. $M=\{\sm{g_1}{}{}{g_2}:g_1,g_2\in G'\}$ and $U=\{\sm{I_n}{X}{}{I_n}:X\in\Mat_{n\times n}(F)\}$.
Let $\varpi$ be the fundamental weight of $P$, i.e., the character of $M$ given by
\[
\varpi(\sm{g_1}{}{}{g_2})=\abs{\frac{\det g_1}{\det g_2}}^{\frac12}, \ g_1,g_2\in G'.
\]
Then, by definition, $\Speh(\pi)$ is the Langlands quotient (i.e., the unique irreducible quotient)
of the induced representation $I_P(\pi\otimes\pi,\varpi)$,
the normalized parabolic induction with respect to $P$ of the irreducible representation
$(\pi\otimes\pi)\cdot\varpi=\pi\cdot\abs{\det}^{\frac12}\otimes\pi\cdot\abs{\det}^{-\frac12}$ of $M$.
Alternatively, $\Speh(\pi)$ is also the unique irreducible subrepresentation of $I_P(\pi\otimes\pi,\varpi^{-1})$,
similarly defined. The representation $\Speh(\pi)$ is unitarizable.

We denote by $\Irr G$ the set of irreducible representations of $G$ and by $\Irr_{\temp}G'$ (resp., $\Irr_{\disc}G'$) the set of irreducible tempered (resp., discrete series)
representations of $G'$, up to equivalence.
Fix a Haar measure $dg$ for $G'$.
Let $\mu_{\pl}$ be the Plancherel measure on $\Irr_{\temp}G'$ \cite{MR1989693},
characterized by the relation
\[
f(e)=\int_{\Irr_{\temp}G'}\tr\pi(f)\ d\mu_{\pl}(\pi),\ \ f\in\swrz(G').
\]

Let $\pi\in\Irr_{\temp}G'$ and $\sigma=\Speh(\pi)\in\Irr G$.
We will define below a realization $\model_{\psi_N}(\sigma)$ of $\sigma$
with an explicit invariant inner product and a non-trivial $H$-invariant linear form $\ell_H$.
We note that the space of $H$-invariant linear forms on $\sigma$ is one-dimensional.

For any $f_1,f_2\in\swrz(G)$ let
\[
(f_1,f_2)_\sigma=\sum_v\ell_H(\sigma(f_1)v)\overline{\ell_H(\sigma(f_2)v)}
\]
where $v$ ranges over a suitable orthonormal basis of $\model_{\psi_N}(\sigma)$.
Since $\ell_H$ is $H$-invariant, the positive semi-definite hermitian form $(f_1,f_2)_\sigma$ factors through
the canonical map $\swrz(G)\rightarrow\swrz(H\bs G)$.
We continue to denote the resulting form on $\swrz(H\bs G)$ by $(\cdot,\cdot)_\sigma$.

Our main result is the following.

\begin{theorem} \label{thm: main}
For a suitable choice of Haar measures and for any $\varphi_1,\varphi_2\in\swrz(H\bs G)$ we have
\begin{equation} \label{eq: inner}
(\varphi_1,\varphi_2)_{L^2(H\bs G)}=\int_{\Irr_{\temp}(G')}(\varphi_1,\varphi_2)_{\Speh(\pi)}\ d\mu_{\pl}(\pi)
\end{equation}
where the right-hand side is an absolutely convergent integral.
\end{theorem}

\begin{corollary} \label{cor: main}
We have the following decomposition of unitary representations of $G$:
\[
L^2(H\bs G)\simeq \int_{\Irr_{\temp}(G')}\Speh(\pi)\ d\mu_{\pl}(\pi).
\]
In particular, an irreducible representation $\sigma$ of $G$ is relatively discrete series with respect to
$H\bs G$ if and only if $\sigma=\Speh(\pi)$ for some $\pi\in\Irr_{\disc}(G')$.
\end{corollary}

Recall that a representation $\sigma\in\Irr G$ is called relatively discrete series if it has a unitary central character
and it admits a non-trivial $H$-invariant form $\ell$ such that for any $v$ in the space of $\sigma$, the matrix coefficients
$\ell(\sigma(g)v)$ is square-integrable on $ZH\bs G$ where $Z$ is the center of $G$. Equivalently, $\sigma$ occurs discretely in
the space $L^2(ZH\bs G;\omega_\sigma)$ of left $H$-invariant functions on $G$ that are $Z$-equivariant
under the central character $\omega_\sigma$ of $\sigma$ and are square-integrable on $ZH\bs G$.

We remark that the fact that $\Speh(\pi)$ is relatively discrete series for any $\pi\in\Irr_{\disc}(G')$
had been proved by Jacquet (unpublished) and independently by Smith \cite{1812.04091}.

\subsection{}
We now describe the abovementioned explicit realization of $\sigma=\Speh(\pi)$,
together with the inner product and the $H$-invariant functional in this realization.

We will use the Zelevinsky model of $\sigma$. (See \cite{MR584084}*{\S8}, where the terminology ``degenerate Whittaker model'' is used.)
More precisely, let $N$ be the maximal unipotent subgroup of $G$ consisting of upper unitriangular matrices.
Let $\psi_N$ be a character on $N$ that is trivial on $U$ and whose restriction to $N\cap M$ is non-degenerate.
Then, up to a constant there exists a unique $(N,\psi_N)$-equivariant functional on $\sigma$,
and this gives rise to a unique realization $\model_{\psi_N}(\sigma)$ of $\sigma$ in the space of left $(N,\psi_N)$-equivariant
functions on $G$.

The inner product on $\model_{\psi_N}(\sigma)$ is defined as follows.
Let $D$ be the joint stabilizer of $e_n$ and $e_{2n}$ in $G$, i.e., the subgroup of matrices in $G$
whose $n$-th and $2n$-th row are $e_n$ and $e_{2n}$ respectively.
Then, it was proved in \cite{1806.10528} that the integral
\[
[W_1,W_2]=\int_{D\cap N\bs D}W_1(g)\overline{W_2(g)}\ dg,\ \ W_1,W_2\in\model_{\psi_N}(\sigma)
\]
converges and is $G$-invariant. (The result in [ibid.] is in fact for an arbitrary Speh representation.)
Note that the working assumption in [ibid.] was that $F$ is of characteristic $0$.
However, this is inessential, as explained in the Appendix by Dmitry Gourevitch.

To define the $H$-invariant functional on $\model_{\psi_N}(\sigma)$, we assume in addition that
$\psi_N$ is trivial on $N\cap H$. Let $\mira$ be the mirabolic subgroup of $G$ (the stabilizer of $e_{2n}$).
Then, as we prove in \S\ref{sec: zel2sym}, the integral
\[
\ell_H(W)=\int_{N\cap H\bs \mira\cap H}W(h)\ dh,\ \ W\in\model_{\psi_N}(\sigma)
\]
converges and is $H$-invariant.

Note that since the characters of $N$ that are trivial on $U(N\cap H)$ and are non-degenerate on $N\cap M$
form a single $T\cap H$-orbit (where $T$ is the diagonal torus of $G$), the veracity of Theorem \ref{thm: main}
is independent of the choice of $\psi_N$. Thus, we are free to choose $\psi_N$.
Of course, the choice of Haar measures will depend on $\psi_N$.

\subsection{}
Theorem \ref{thm: main} and its proof are modeled in part after the recent remarkable paper \cite{1812.00047} of Beuzart-Plessis.
In fact, our case is simpler since we do not need the intricate limit analysis of
\cite{1812.00047}*{\S3}. This has to do with the fact that the functoriality in  \cite{1812.00047}
is base change while in our case it is just taking Langlands quotient.
On the flip side, the Plancherel measure of $L^2(H\bs G)$ is supported off the tempered spectrum of $G$,
and this creates additional technical difficulties.

The main new input is an identity described in Theorem \ref{thm: cmprbes} below, which is
a local analogue of \cite{MR2058616}*{Theorem 4}.
It is based on two ingredients.
The first is a relation, proved in \cite{1806.10528}*{Appendix A} between the inner product
on $\model_{\psi_N}(\sigma)$ and the standard invariant pairing between $I_P(\pi\otimes\pi,\varpi)$
and $I_P(\tilde\pi\otimes\tilde\pi,\varpi^{-1})$, where $\tilde\pi$ is the contragredient of $\pi$.
The second is a relation (essentially a local analogue of \cite{MR2058616}*{Theorem 2})
between $\ell_H$ and an $H$-invariant functional on $I_P(\pi\otimes\pi,\varpi)$ which
is a local analogue of the one considered in \cite{MR1142486}.
Both relations involve the standard intertwining operator from $I_P(\pi\otimes\pi,\varpi)$ to
$I_P(\pi\otimes\pi,\varpi^{-1})$.

\subsection{}
We also have a variant of Theorem \ref{thm: main} and Corollary \ref{cor: main} for the symplectic similitude group
$\tilde H=\GSp_n(F)$.
The point is that the invariant functional $\ell_H$ is $(\tilde H,\omega_\pi)$-equivariant
where a character of $F^*$ is viewed as a character of $\tilde H$ via the similitude factor.
Fix a unitary character $\chi$ of $F^*$.
Let $\Irr_{\temp}^\chi(G')\subset\Irr_{\temp}(G')$ be the subset of tempered representations with central character $\chi$
and let $\mu_{\pl}^\chi$ be the Plancherel measure on $\Irr_{\temp}^\chi G'$,
characterized by the relation
\[
f(e)=\int_{\Irr_{\temp}^\chi G'}\tr\pi(f)\ d\mu_{\pl}^\chi(\pi),\ \ f\in\swrz(Z'\bs G';\chi^{-1}),
\]
where $Z'$ is the center of $G'$ and $\pi(f)=\int_{Z'\bs G'}f(g)\pi(g)\ dg$.
(Here, $\swrz(Z'\bs G';\chi^{-1})$ denotes the space of locally constant functions $f$ on $G$ such that
$f(zg)=\chi(z)^{-1}f(g)$ for all $z\in Z'$, $g\in G'$ and $f$ is compactly supported modulo $Z'$.)
Equivalently, for any smooth, compactly supported function $h$ on $\Irr_{\temp} G'$ we have
\begin{equation} \label{eq: defmuchi}
\int_{F^*}\big(\int_{\Irr_{\temp} G'} h(\pi)\omega_\pi^{-1}(z)\ d\mu_{\pl}(\pi)\big)\chi(z)\ dz=
\int_{\Irr_{\temp}^\chi G'} h(\pi)\ d\mu_{\pl}^\chi(\pi),
\end{equation}
where the left-hand side converges as an iterated integral.

For any $\varphi\in\swrz(H\bs G)$ let $\tilde\varphi(g)=\int_{H\bs\tilde H}\varphi(tg)\chi(t)\ dt$.
Thus, $\varphi\mapsto\tilde\varphi$ defines a projection $\swrz(H\bs G)\rightarrow\swrz(\tilde H\bs G;\chi^{-1})$.
For any $\pi\in\Pi_{\temp}^\chi$ the positive semi-definite hermitian form $(\varphi_1,\varphi_2)_{\Speh(\pi)}$ depends only on
$\tilde\varphi_i$, $i=1,2$. We continue to denote the resulting form on $\swrz(\tilde H\bs G;\chi^{-1})$
by $(\cdot,\cdot)_{\Speh(\pi)}$.
Denote by $L^2(\tilde H\bs G;\chi^{-1})$ the Hilbert space of $(\tilde H,\chi^{-1})$-equivariant functions
that are square-integrable modulo $\tilde H$.

\begin{theorem} \label{thm: gsp}
For any $\varphi_1,\varphi_2\in\swrz(\tilde H\bs G;\chi^{-1})$ we have
\[
(\varphi_1,\varphi_2)_{L^2(\tilde H\bs G;\chi^{-1})}=\int_{\Irr_{\temp}(G')^\chi}(\varphi_1,\varphi_2)_{\Speh(\pi)}\ d\mu_{\pl}^\chi(\pi).
\]
Then we have a decomposition of Hilbert $G$-representations
\[
L^2(\tilde H\bs G;\chi^{-1})\simeq \int_{\Irr_{\temp}^\chi(G')}\Speh(\pi)\ d\mu_{\pl}^\chi(\pi).
\]
\end{theorem}

\subsection{Notation}
We introduce some more notation that will be used throughout.
\begin{itemize}
\item For a locally compact group $X$ we denote by $\modulus_X$ its modulus function and
$R(\cdot)$ and $L(\cdot)$ the right and left regular representations of $X$ on itself.
\item The groups $G, P=M\ltimes U, N, \mira, D, G', Z', H,\tilde H$ are as above.
\item Let $\K=\GL_{2n}(\OOO)$ be the standard maximal compact subgroup of $G$.
\item Let $\crnr:G'\rightarrow G$ be the embedding $g\mapsto\sm g{}{}{I_n}$.
\item We denote by $\tilde\lambda$ the similitude character of $\tilde H$ and by $\tilde\lambda^\vee$
the cocharacter $\tilde\lambda^\vee(a)=\crnr(aI_n)$.
\item Let $\mira'$ be the mirabolic subgroup of $G'$, i.e., the stabilizer of $e_n$.
\item For any subgroup $X$ of $G$ we write $X_H=X\cap H$ and $X_{\tilde H}=X\cap\tilde H$.
\item In particular, $P_H=M_H\ltimes U_H$ is the Siegel parabolic subgroup of $H$ with its standard Levi decomposition.
\item Likewise, $P_{\tilde H}=M_{\tilde H}U_{\tilde H}$ is a maximal parabolic subgroup of $\tilde H$,
$M_{\tilde H}=M_H\times\crnr(Z')$ and $U_{\tilde H}=U_H$.
\item We denote by $\embd:G'\rightarrow M_H$ the isomorphism given by
\[
\embd(g)=\sm{g^*}{}{}g,\text{ where }g^*=w_n\,^tg^{-1}w_n.
\]
\item A basic fact is that
\begin{equation} \label{eq: mod}
\modulus_{P_{\tilde H}}=\modulus_P^{\frac12}\varpi\text{ on }M_{\tilde H}\text{ and in particular, } \modulus_{P_H}=\modulus_P^{\frac12}\varpi\text{ on }M_H.
\end{equation}
\item We also write $H_{n-1}$ for the group $\Sp_{n-1}(F)$ viewed as a subgroup of $H$ via $h\mapsto\left(\begin{smallmatrix}1&&\\&h&\\&&1\end{smallmatrix}\right)$.
\item Let $w_U=\sm{}{I_n}{-I_n}{}\in H$, an element that normalizes $M$ and $M_H$.
\item We fix once and for all a non-trivial character $\psi$ of $F$.
\item Let $N'$ be the subgroup of unitriangular matrices in $G'$ and let $\psi_{N'}$ be the non-degenerate character on $N'$ given by
\[
\psi_{N'}(g)=\psi(g_{1,2}+\dots+g_{n-1,n}).
\]
\item Let $N_M=N\cap M\simeq N'\times N'$ and let $\psi_{N_M}$ be the non-degenerate character on $N_M$ given by
\[
\psi_{N_M}(\sm{g_1}{}{}{g_2})=\psi_{N'}(g_1g_2).
\]
\item The character $\psi_N$ on $N$ is the one that is trivial on $U$ and restricts to $\psi_{N_M}$ on $N_M$.
\item The convention of Haar measures will be as in \cite{1812.00047}*{\S2.5}.
The character $\psi$ gives rise to a Haar measure on $F$ which is self-dual with respect to $\psi$.
In turn, this gives rise to a Haar measure on $X(F)$ for any linear algebraic group $X$ over $F$ with a $\Z$-model
for $X_{\bar F}$. In particular, if $X$ is reductive, then we get a canonical measure.
For instance, for $G$ itself the measure is $\abs{\det g}^{-2n}\prod_{i,j}dg_{i,j}$
where $g_{i,j}$ are the coordinates of $g$.
All the abovementioned algebraic subgroups of $G$ and $H$ have ``obvious'' $\Z$-models,
as they are defined by vanishing of coordinates in $G$ or in $H$.
This will be implicitly used to define Haar measure on them.
\item We will write $A\ll B$ to signify that $A$ is bounded by a constant multiple of $B$.
If the implied constant depends on an additional parameter, say $x$, we will write $A\ll_x B$.
\end{itemize}

\section{An identity of Bessel distributions} \label{sec: besidn}

\subsection{}
Let $V$ and $V^\vee$ be two admissible smooth representations of an $\ell$-group $X$ and let
\[
\bil:V\times V^\vee\rightarrow\C
\]
be an $X$-invariant non-degenerate bilinear form.
Thus, $\bil$ defines an isomorphism between $V^\vee$ and the smooth dual of $V$.
We refer to $(V,V^\vee,\bil)$ with the group action as $X$-representations in duality.
(Normally, the group action will be clear from the context so for simplicity we do not include it
in the notation.)

A morphism
\[
(\Phi,\Phi^\vee):(V_1,V_1^\vee,\bil_1)\rightarrow(V_2,V_2^\vee,\bil_2)
\]
of $X$-representations in duality is a pair of intertwining operators
$\Phi:V_1\rightarrow V_2$, $\Phi^\vee:V_2^\vee\rightarrow V_1^\vee$ such that
$\bil_2(\Phi v_1,v_2^\vee)=\bil_1(v_1,\Phi^\vee v_2^\vee)$ for all $v_1\in V_1$, $v_2\in V_2^\vee$.

In general, we denote by $\tilde\pi$ the contragredient of a representation $\pi$ (i.e., the smooth
dual of $\pi$).

We will consider the following examples pertaining to an irreducible tempered representation
$\pi$ of $G'=\GL_n(F)$.

\begin{example}
Let $\model_{\psi_{N'}}(\pi)$ be the Whittaker model of $\pi$ with respect to $\psi_{N'}$, with right translation.
Define a bilinear form
\[
\bil_{\mira'}(W,W^\vee)=\int_{N'\bs\mira'}W(g)W^\vee(g)\ dg
\]
on $\model_{\psi_{N'}}(\pi)\times\model_{\psi_{N'}^{-1}}(\tilde\pi)$.
The integral is absolutely convergent and by a well-known result of Bernstein,
$\bil_{\mira'}$ is $G'$-invariant \cite{MR748505}. Thus, the tuple
\[
\data_{\psi_{N'}}(\pi)=(\model_{\psi_{N'}}(\pi),\model_{\psi_{N'}^{-1}}(\tilde\pi),\bil_{\mira'})
\]
is $G'$-representations in duality. Moreover,
\[
\model_{\psi_{N'}^{-1}}(\tilde\pi)=\{W^*:W\in\model_{\psi_{N'}}(\pi)\}\ \ \text{where }W^*(g)=W(g^*).
\]

Similarly, we may also consider
\[
\data_{\psi_{N_M}}(\pi\otimes\pi)=(\model_{\psi_{N_M}}(\pi\otimes\pi),\model_{\psi_{N_M}^{-1}}(\tilde\pi\otimes\tilde\pi),
\bil_{\mira'\times\mira'}=\bil_{\mira'}\otimes\bil_{\mira'})
\]
of $M$-representations in duality.
Note that $\model_{\psi_{N_M}}(\pi\otimes\pi)=\model_{\psi_{N'}}(\pi)\otimes\model_{\psi_{N'}}(\pi)$ and similarly for
$\model_{\psi_{N_M}^{-1}}(\tilde\pi\otimes\tilde\pi)$.
\end{example}

\begin{example}
Let $\data=(V,V^\vee,\bil)$ be $M$-representations in duality. Denote the representation on $V$ by $\tau$.
For any character $\chi$ of $M$ consider the induced representation $I_P(V,\chi)$
realized in the space of functions $\varphi:G\rightarrow V$ such that
\begin{equation} \label{eq: rtrns}
\varphi(mg)=\modulus_P^{\frac12}(m)\chi(m)\tau(m)(\varphi(g))\ \ \forall m\in M, g\in G.
\end{equation}
The action of $G$ is by right translation: $I_P(g,\chi)\varphi(x)=\varphi(xg)$.

Then, we have a $G$-duality data
\[
I_P(\data,\chi)=(I_P(V,\chi),I_P(V^\vee,\chi^{-1}),\bil_{P\bs G})
\]
where
\[
\bil_{P\bs G}(\varphi,\varphi^\vee)=\int_{P\bs G}\bil(\varphi(g),\varphi^\vee(g))\ dg.
\]
\end{example}

\begin{example}
Let $\sigma=\Speh(\pi)\in\Irr G$. Note that $\tilde\sigma=\Speh(\tilde\pi)$. Consider
the Zelevinsky model $\model_{\psi_N}(\sigma)$ of $\sigma$.
Recall that $D$ is the joint stabilizer of $e_n$ and $e_{2n}$.
In \cite{1806.10528} the bilinear form
\[
\bil_D(W,W^\vee)=\int_{D\cap N\bs D}W(g)W^\vee(g)\ dg
\]
on $\model_{\psi_N}(\sigma)\times\model_{\psi_N^{-1}}(\tilde\sigma)$ is introduced and it is shown to be $G$-invariant.
We write
\[
\data_{\psi_N}(\sigma)=(\model_{\psi_N}(\sigma),\model_{\psi_N^{-1}}(\tilde\sigma),\bil_D)
\]
for the $G$-representations in duality.
\end{example}

Note that $\model_{\psi_N^{-1}}(\tilde\sigma)$ can be considered as a subrepresentation of $I_P(\model_{\psi_{N_M}^{-1}}(\tilde\pi\otimes\tilde\pi),\varpi^{-1})$
via the embedding
\[
[\iota^\vee(W)(g)](m)=\modulus_P(m)^{-\frac12}\varpi(m)W(mg),\ m\in M,\ g\in G.
\]

Let
\[
\M_\pi:I_P(\model_{\psi_{N_M}}(\pi\otimes\pi),\varpi)\rightarrow I_P(\model_{\psi_{N_M}}(\pi\otimes\pi),\varpi^{-1})
\]
be the intertwining operator given by the absolutely convergent integral
\[
[\M_\pi\varphi(g)](x)=\int_U\varphi(w_U^{-1}ug)(w_U^{-1}xw_U)\ du,\ \ g\in G, x\in M.
\]
(Recall that $w_U=\sm{}{I_n}{-I_n}{}\in H$.)
Let $\tilde\M_\pi$ be the composition of $\M_\pi$ with the injective map
$I_P(\model_{\psi_{N_M}}(\pi\otimes\pi),\varpi^{-1})\rightarrow C^\infty(N\bs G,\psi_N)$ given by $\varphi\mapsto \varphi(g)(e)$.
Then, the image of $\tilde\M_\pi$ is $\model_{\psi_N}(\sigma)$.

By \cite{1806.10528}*{Appendix A}, the pair $(\tilde\M_\pi,\iota^\vee)$ is a morphism of $G$-representations in duality
\[
I_P(\data_{\psi_{N_M}}(\pi\otimes\pi),\varpi)\rightarrow\data_{\psi_N}(\Speh(\pi)).
\]

\subsection{}
Suppose that $\data=(V,V^\vee,\bil)$ is admissible $X$-representations in duality.
Denote by $\pi$ (resp. $\pi^\vee$) the corresponding representation of $X$ on $V$ (resp. $V^\vee$).
Using $\bil$ we may identify $V\otimes V^\vee$ with the space $\End_{\smth}(V)$ of smooth linear endomorphisms of $V$
(i.e., the linear maps $A:V\rightarrow V$ such that $A\pi(g)=\pi(g)A=A$ for all $g$ in a sufficiently small
open subgroup of $X$).
In particular, (fixing a Haar measure on $X$) for any $f\in\swrz(X)$ we view $\pi(f)$ as an element of $V\otimes V^\vee$.
For any $\ell\in V^*$ (the algebraic dual of $V$) and $\ell^\vee\in (V^\vee)^*$ define
\[
\Bes_{\data}^{\ell,\ell^\vee}(f)=(\ell\otimes\ell^\vee)[\pi(f)]=\ell^\vee(\ell\circ\pi(f))
\]
(where the smooth functional $\ell\circ\pi(f)$ on $V$ is viewed as an element of $V^\vee$ via $\bil$).

Let $\basis$ be a basis for $V$ and let $K$ be an open subgroup of $X$.
We say that $\basis$ is compatible with $K$ if
for every $v\in\basis$, either $v$ is $K$-invariant or $\int_K\pi(k)v\ dk=0$.
We say that $\basis$ is \emph{admissible} if it is compatible with a family of open subgroups of $X$ that
form a neighborhood base for the identity.

To construct an admissible basis, fix a compact open subgroup $K$ of $X$ and for each
irreducible representation $\tau$ of $K$ take a basis $\basis_\tau$ for the $\tau$-isotypic part of $V$.
Then, $\basis=\cup_\tau\basis_\tau$ is compatible with any normal open subgroup of $K$.

If $\basis$ is an admissible basis, then we can form the admissible dual basis $\basis^\vee$ for $V^\vee$
with a bijection $\tilde{}:\basis\rightarrow\basis^\vee$ such that
\[
\bil(v,\tilde u)=\delta_{u,v}\ \ u,v\in\basis.
\]
Indeed, for each $v\in\basis$ define $\tilde v\in V^*$ by $\bil(u,\tilde v)=\delta_{u,v}$ for all $u\in\basis$.
Then, it is easy to see that if $v$ is $K$-invariant and $\basis$ is compatible with $K$,
then $\tilde v$ is $K$-invariant. Moreover, if $\basis$ is compatible with $K$, then so is $\basis^\vee$.

Some elementary facts about Bessel distributions are itemized in the following. (We assume for simplicity that $X$ is unimodular.)
\begin{lemma} \label{lem: beselem}
\begin{enumerate}
\item We have
\[
\Bes_{\data}^{\ell,\ell^\vee}(f)=\Bes_{\data^\circ}^{\ell^\vee,\ell}(f^\vee)
\]
where $\data^\circ=(V^\vee,V,\bil^\circ)$, $\bil^\circ(v^\vee,v)=\bil(v,v^\vee)$ and $f^\vee(g)=f(g^{-1})$.
\item Suppose that $\basis$ is an admissible basis for $V$ and let $\basis^\vee$ be the dual basis for $V^\vee$.
Then,
\begin{equation} \label{eq: besOB}
\Bes_{\data}^{\ell,\ell^\vee}(f)=\sum_{v\in\basis}\ell(\pi(f)v)\ell^\vee(\tilde v)
\end{equation}
where only finitely many terms in the sum are non-zero.
\item For any $g_1,g_2\in X$ we have
\[
\Bes_{\data}^{\ell,\ell^\vee}(L(g_1)R(g_2)f)=\Bes_{\data}^{\ell\circ\pi(g_1),\ell^\vee\circ\pi^\vee(g_2)}(f).
\]
\item \label{part: inter}
If $(\Phi,\Phi^\vee):\data_1=(V_1,V_1^\vee,\bil_1)\rightarrow\data_2=(V_2,V_2^\vee,\bil_2)$ is a homomorphism of
representations in duality, then for any $\ell\in V_2^*$ and $\ell^\vee\in (V_1^\vee)^*$ we have
\[
\Bes_{\data_2}^{\ell,\ell^\vee\circ\Phi^\vee}(f)=\Bes_{\data_1}^{\ell\circ\Phi,\ell^\vee}(f).
\]
\item \label{part: convo} View $\bil$ as a linear form on $V\otimes V^\vee$.
Then, for any $f\in\swrz(X\times X)$ we have
\[
\Bes_{\data\otimes\data^\circ}^{\bil,\ell^\vee\otimes\ell}(f)=
\Bes_{\data}^{\ell,\ell^\vee}(f')
\]
where $f'\in\swrz(X)$ is given by
\[
f'(x)=\int_X f(yx,y)\ dy,\ x\in X.
\]
\end{enumerate}
\end{lemma}

\begin{proof}
Let $s:V\otimes V^\vee\rightarrow V^\vee\otimes V$ be given by $v\otimes v^\vee\mapsto v^\vee\otimes v$.
Then, identifying $V\otimes V^\vee$ with $\End_{\smth}(V)$ and $V^\vee\otimes V$ with $\End_{\smth}(V^\vee)$,
$s$ takes $A\in\End_{\smth}(V)$ to its adjoint. In particular, it takes the operator $\pi(f)$
to $\pi^\vee(f^\vee)$. The first part follows.

For the second part, suppose that $K$ is an open subgroup of $X$ such that $f$ is bi-$K$-invariant
and $\basis$ is compatible with $K$.
The space $V^K$ is finite dimensional and admits $\basis^K:=\basis\cap V^K$ as a basis.
The spaces $V^K$ and $(V^K)^\vee$ are in duality with respect to $\bil$.
Let $u^\vee\in (V^\vee)^K$ be such that $\ell(v)=\bil(v,u^\vee)$ for all $v\in V^K$
and let $u\in V^K$ be such that $\ell^\vee(v^\vee)=\bil(u,v^\vee)$ for all $v^\vee\in (V^\vee)^K$.
Then,
$\Bes_{\data}^{\ell,\ell^\vee}(f)$ is the trace of the composition of $\pi(f)$ on $V^K$
with the rank-one operator $A(v)=\bil(v,u^\vee)u$ on $V^K$.
The adjoint $A^\vee$ of $A$ is given by $A^\vee(v^\vee)=\bil(u,v^\vee)u^\vee$. Thus,
\[
\Bes_{\data}^{\ell,\ell^\vee}(f)=\tr A\pi(f)=\sum_{v\in\basis^K}\bil(\pi(f)v,A^\vee\tilde v)=
\sum_{v\in\basis^K}\ell(\pi(f)v)\ell^\vee(\tilde v).
\]
On the other hand, $\pi(f)v=0$ for all $v\in\basis\setminus\basis^K$ since $\basis$ is compatible with $K$.

The third part is immediate from the definition.

For the fourth part, we have
\[
\Bes_{\data_2}^{\ell,\ell^\vee\circ\Phi^\vee}(f)=
\ell^\vee\circ\Phi^\vee(\ell\circ\pi_2(f))=
\ell^\vee(\ell\circ\pi_2(f)\circ\Phi)=\ell^\vee(\ell\circ\Phi\circ\pi_1(f))=
\Bes_{\data_1}^{\ell\circ\Phi,\ell^\vee}(f).
\]

For the last part we may write using the second part
\[
\Bes_{\data\otimes\data^\circ}^{\bil,\ell^\vee\otimes\ell}(f)=
\sum_{u,v\in\basis}\bil((\pi\otimes\pi^\vee)(f)(u\otimes\tilde v))
(\ell^\vee\otimes\ell)(\tilde u\otimes v).
\]
Note that
\begin{align*}
&\bil((\pi\otimes\pi^\vee)(f)(u\otimes\tilde v))=
\int_{X\times X}f(x,y)\bil((\pi\otimes\pi^\vee)(x,y)(u\otimes\tilde v))\ dx\ dy
\\=&\int_{X\times X}f(x,y)\bil(\pi(x)u,\pi^\vee(y)\tilde v)\ dx\ dy=
\int_X\int_Xf(yx,y)\bil(\pi(y)\pi(x)u,\pi^\vee(y)\tilde v)\ dx\ dy
\\=&\int_X\int_Xf(yx,y)\bil(\pi(x)u,\tilde v)\ dx\ dy=
\int_Xf'(x)\bil(\pi(x)u,\tilde v)\ dx=\bil(\pi(f')u,\tilde v).
\end{align*}
Since
\[
\pi(f')u=\sum_{v\in\basis}\bil(\pi(f')u,\tilde v)v
\]
we get
\[
\Bes_{\data\otimes\data^\circ}^{\bil,\ell^\vee\otimes\ell}(f)=
\sum_{u,v\in\basis}\bil(\pi(f')u,\tilde v)\ell(v)\ell^\vee(\tilde u)=
\sum_{u\in\basis}\ell(\pi(f')u)\ell^\vee(\tilde u)=\Bes_{\data}^{\ell,\ell^\vee}(f')
\]
as required.
\end{proof}

\subsection{}

We now state the main result of this section, which will be proved in the rest of this and the next section.

Let $\pi\in \Irr_{\temp} G'$ and $\ell_H$ be the linear form on $\model_{\psi_N}(\Speh(\pi))$ given by
\begin{align*}
\ell_H(W)=\int_{N_H\bs \mira_H}W(h)\ dh=&\int_{P_H\cap \mira_H\bs \mira_H}\int_{N_H\bs P_H\cap \mira_H}
\modulus_{P_H\cap \mira_H}(q)^{-1}W(qh)\ dq\ dh
\\=&\int_{P_H\cap \mira_H\bs \mira_H}\int_{N'\bs\mira'}\abs{\det m}^nW(\embd(m)h)\ dm\ dh.
\end{align*}
The convergence of the integral will be proved in Lemma \ref{lem: abs conv} below.

\begin{theorem} \label{thm: cmprbes}
Let $\pi\in\Irr_{\temp}(G')$ and $\sigma=\Speh(\pi)$. Then, we have
\[
\Bes_{\data_{\psi_N}(\sigma)}^{\ell_H,\dirac_e}(f)
=\Bes_{\data_{\psi_{N'}}(\pi)}^{\dirac_e,\dirac_e}(\trns f),\ \ f\in\swrz(G)
\]
where $\dirac_e$ is the evaluation at the identity (in the Zelevinsky/Whittaker model)
and $\trns f\in\swrz(G')$ is given by
\begin{equation} \label{def: trns}
\trns f(g)=\abs{\det g}^{\frac{1-n}2}\int_{U_H\bs U}\varphi(u\crnr(g))\ du,\ \ g\in G'
\end{equation}
where $\varphi=\int_Hf(h\cdot)\ dh$.
\end{theorem}

\subsection{}
It follows from \eqref{eq: mod} that for any $M_H$-invariant linear form $\lambda$ on $\pi\otimes\pi$, the integral
\[
\int_{P_H\bs H}\lambda(\varphi(h))\ dh
\]
defines an $H$-invariant linear form on $I_P(\pi\otimes\pi,\varpi)$.
By \cite{MR1078382}*{Theorem 2.4.2} and the proof of \cite{MR2248833}*{Proposition 6},
this construction gives rise to a linear isomorphism of one-dimensional vector spaces
\[
\Hom_{M_H}(\pi\otimes\pi,\triv)\simeq\Hom_H(I_P(\pi\otimes\pi,\varpi),\triv)=\Hom_H(\Speh(\pi),\triv).
\]

By using \cite{MR748505} once again, a concrete $M_H$-invariant linear form on $\model_{\psi_{N_M}}(\pi\otimes\pi)$ is given by
\[
\ell_{M_H}(W)=\int_{N'\bs\mira'} W(\embd(p))\ dp.
\]
Hence, the linear form
\begin{equation} \label{def: ellHind}
\ell_H^{\ind}(\varphi)=\int_{P_H\bs H} \ell_{M_H}(\varphi(h))\ dh,\ \ \varphi\in I_P(\model_{\psi_{N_M}}(\pi\otimes\pi),\varpi)
\end{equation}
is well-defined and $H$-invariant.

Let
\[
\ell_{M_H}^*(W)=\ell_{M_H}(W^{w_U}),\ \ \ W\in \model_{\psi_{N_M}}(\pi\otimes\pi)
\]
where $W^{w_U}(x)=W(w_Uxw_U^{-1})$.
Then, $\ell_{M_H}^*$ is also $M_H$-invariant, and in fact $\ell_{M_H}^*\equiv\ell_{M_H}$
since they must differ by a sign, and they are positive on $W$ of the form
$W(\sm{g_1}{}{}{g_2})=W_1(g_1)\overline{W_1(g_2^*)}$ where $0\not\equiv W_1\in\model_{\psi_{N'}}(\pi)$.

We denote by $\model_H(\sigma)$ (the \emph{symplectic model} of $\sigma=\Speh(\pi)$) the image of the map
\[
\transs:I_P(\model_{\psi_{N_M}}(\pi\otimes\pi),\varpi)\rightarrow C^\infty(H\bs G)
\]
(which factors through $\sigma$) induced from $\ell_H^{\ind}$ by Frobenius reciprocity, i.e.,
\begin{equation} \label{def: transs}
\transs\varphi(g)=\ell_H^{\ind}(I_P(g,\varpi)\varphi).
\end{equation}

\begin{remark} \label{rem: symp}
Recall that $\tilde H$ is the symplectic similitude group and $\tilde\lambda$ is the similitude character of $\tilde H$.
Then, for any $\varphi\in I_P(\model_{\psi_{N_M}}(\pi\otimes\pi),\varpi)$ we have
\[
\ell_H^{\ind}(\varphi)=\int_{P_{\tilde H}\bs\tilde H} \ell_{M_H}(\varphi(\tilde h))
\omega_\pi(\tilde\lambda(\tilde h))^{-1}\ d\tilde h,
\]
and hence, $\ell_H^{\ind}$ is $(\tilde H,\omega_\pi\circ\tilde\lambda)$-equivariant.
Indeed, since $P_H\bs H\simeq P_{\tilde H}\bs\tilde H$, it is enough to check that the integral
on the right-hand side is well-defined, i.e., that the integrand has the correct equivariance property.
This follows from \eqref{eq: mod} and the fact that $P_{\tilde H}=\tilde\lambda^\vee(F^*)P_H$
and $\tilde\lambda(\tilde\lambda^\vee(a))=a$.
\end{remark}

\begin{proposition} \label{prop: mainper}
We have
\begin{equation} \label{eq: LHellH}
\ell_H^{\ind}=\ell_H\circ\tilde\M_{\pi}.
\end{equation}
\end{proposition}

We will prove the proposition in \S\ref{sec: modelfin} after some preparation.

For now, we give a \emph{heuristic} argument for the validity of Proposition \ref{prop: mainper}
in the spirit of the (rigorous) argument of \cite{1806.10528}*{Appendix A}.
For this argument we \emph{assume} that every $\mira_H$-invariant distribution on $H\bs G$ is $H$-invariant.
This is an expected (but as yet, unproved) analogue of \cite{MR748505}.
It would imply that for any $H$-distinguished irreducible representation of $G$, every $\mira_H$-invariant functional is $H$-invariant.
In particular, $\ell_H$ is $H$-invariant.
Since there is a unique $H$-invariant functional on $I_P(\pi\otimes\pi,\varpi)$ up to a scalar,
we conclude that $\ell_H^{\ind}$ and $\ell_H\circ \tilde\M_{\pi}$ are proportional.
To determine the proportionality constant we compare the integrals
\[
I_1=\int_{\bar U_H\bs\bar U}\ell_H^{\ind}(I_P(u,\varpi)\varphi)\ du\ \ \text{ and }\ \
I_2=\int_{\bar U_H\bs\bar U}\ell_H(\tilde\M_\pi(I_P(u,\varpi)\varphi))\ du
\]
where $\bar X$ denotes the image of $X$ under transpose.

On the one hand,
\[
I_1=\int_{\bar U_H\bs\bar U}\int_{\bar U_H}\ell_{M_H}(\varphi(uv))\ du\ dv=
\int_{\bar U}\ell_{M_H}(\varphi(u))\ du=
\ell_{M_H}^*(\M_\pi\varphi(w_U)).
\]
On the other hand, observe that the embedding $\bar U_{H_{n-1}}\hookrightarrow \bar U_H$ induces an isomorphism
of abelian groups
\[
\bar U_{H_{n-1}}\bs\bar U_D\simeq \bar U_H\bs\bar U.
\]
Arguing formally, we have
\begin{align*}
I_2=&\int_{\bar U_{H_{n-1}}\bs\bar U_D}\ell_H(\tilde\M_\pi(I_P(u,\varpi)\varphi))\ du=
\int_{\bar U_{H_{n-1}}\bs\bar U_D}\int_{N_H\bs \mira_H}(\tilde\M_\pi\varphi)(hu)\ dh\ du\\&=
\int_{\bar U_{H_{n-1}}\bs\bar U_D}\int_{\bar U_{H_{n-1}}}\int_{N'\bs\mira'}
(\tilde\M_\pi\varphi)(\embd(m)vu)\abs{\det m}^n\ dm\ dv\ du
\\=&\int_{\bar U_D}\int_{N'\bs\mira'}(\tilde\M_\pi\varphi)(\embd(m)u)\abs{\det m}^n\ dm\ du
\\=&\int_{N'\bs\mira'}\int_{\bar U_D}(\tilde\M_\pi\varphi)(\embd(m)u)\ du\ \abs{\det m}^n\ dm
\\=&\int_{N'\bs\mira'}\int_{\bar U_D}(\tilde\M_\pi\varphi)(u\embd(m))\ du\ \abs{\det m}^{1-n}\ dm.
\end{align*}
By \cite{1806.10528}*{Lemma A.2}, this is equal to
\[
\int_{N'\bs\mira'}(\tilde\M_\pi\varphi)(w_U\embd(m))\abs{\det m}^{1-n}\ dm
=\int_{N'\bs\mira'}\M_\pi\varphi(w_U)(\embd(m^*))\ dm
=\ell_{M_H}^*(\M_\pi\varphi(w_U)).
\]
We would conclude that $I_1=I_2$ which would imply \eqref{eq: LHellH}.

Unfortunately, we are unable to justify the computation above for $I_2$ (even for special $\varphi$)
since we do not know whether the integrals above converge as double integrals.
Instead, we will prove Proposition \ref{prop: mainper} in a different way using the argument of \cite{MR3431601},
taking into account a simplification suggested by Rapha\"{e}l Beuzart-Plessis.

\subsection{}
Admitting Proposition \ref{prop: mainper}, we infer from Lemma \ref{lem: beselem} part \ref{part: inter},
applied to the morphism $(\tilde\M_\pi,\iota^\vee)$, that
\[
\Bes_{\data_{\psi_N}(\sigma)}^{\ell_H,\dirac_e}(f)=
\Bes_{I_P(\data_{\psi_{N_M}}(\pi\otimes\pi),\varpi)}^{\ell_H^{\ind},\dirac_{e;e}}(f)
\]
where $\dirac_{e;e}(\varphi)=\varphi(e)(e)$.

Note that $\K_H=\K\cap H$ is a maximal compact subgroup of $H$.
We endow $\K_H$ with the Haar measure induced by the Iwasawa decomposition $H=P_H\K_H$.
(It is not a probability measure.)
Our next claim is

\begin{lemma}
Let $f^\sharp\in\swrz(M)$ be given by
\[
f^\sharp(m)=\modulus_P(m)^{-\frac12}\varpi(m)\int_{\K_H}\int_Uf(kum)\ du\ dk=
\modulus_P(m)^{\frac12}\varpi(m)\int_{\K_H}\int_Uf(kmu)\ du\ dk
\]
for any $m\in M$. Then,
\[
\Bes_{I_P(\data_{\psi_{N_M}}(\pi\otimes\pi),\varpi)}^{\ell_H^{\ind},\dirac_{e;e}}(f)=
\Bes_{\data_{\psi_{N_M}}(\pi\otimes\pi)}^{\ell_{M_H},\dirac_e}(f^\sharp).
\]
\end{lemma}

\begin{proof}
For any $g\in G$ let $\Psi(g)\in\model_{\psi_{N_M}^{-1}}(\tilde\pi\otimes\tilde\pi)$ be such that
\[
\bil_{\mira'\times\mira'}(W,\Psi(g))=\ell_{M_H}((\pi\otimes\pi)((R(g)f)^\sharp)W)\ \ \forall W\in\model_{\psi_{N_M}}(\pi\otimes\pi).
\]
Clearly, $\Psi\in I_P(\model_{\psi_{N_M}^{-1}}(\tilde\pi\otimes\tilde\pi),\varpi^{-1})$.
As in \cite{MR2058616}*{Lemma 2}, we observe that
\[
\ell_H^{\ind}(I_P(f,\varpi)\varphi)=\bil_{P\bs G}(\varphi,\Psi).
\]
Indeed,
\begin{align*}
&\ell_H^{\ind}(I_P(f,\varpi)\varphi)=\int_{P_H\bs H}\int_Gf(g)\ell_{M_H}(\varphi(xg))\ dg\ dx
=\int_{P_H\bs H}\int_Gf(x^{-1}g)\ell_{M_H}(\varphi(g))\ dg\ dx
\\=&\int_{\K_H}\int_{P\bs G}\int_M\int_U\modulus_P(m)^{-\frac12}\varpi(m)f(k^{-1}umg)\ell_{M_H}(R(m)\varphi(g))\ du\ dm\ dg\ dk
\\=&\int_{P\bs G}\ell_{M_H}((\pi\otimes\pi)((R(g)f)^\sharp)\varphi(g))\ dg
=\int_{P\bs G}\bil_{\mira'\times\mira'}(\varphi(g),\Psi(g))\ dg=\bil_{P\bs G}(\varphi,\Psi).
\end{align*}
Thus,
\[
\Bes_{I_P(\data_{\psi_{N_M}}(\pi\otimes\pi),\varpi)}^{\ell_H^{\ind},\dirac_{e;e}}(f)=
\dirac_{e;e}(\Psi)=\dirac_e(\Psi(e))=
\Bes_{\data_{\psi_{N_M}}(\pi\otimes\pi)}^{\ell_{M_H},\dirac_e}(f^\sharp)
\]
as claimed.
\end{proof}

\subsection{}
Finally,
\begin{lemma}
For any $f\in\swrz(M)$ we have
\[
\Bes_{\data_{\psi_{N_M}}(\pi\otimes\pi)}^{\ell_{M_H},\dirac_e}(f)
=\Bes_{\data_{\psi_{N'}}(\pi)}^{\dirac_e,\dirac_e}(\Conv(f))
\]
where $\Conv(f)\in\swrz(G')$ is given by
\[
\Conv(f)(g)=\int_{M_H}f(m\crnr(g))\ dm,\ \ g\in G'.
\]
\end{lemma}

\begin{proof}
Let
\[
\data_{\psi_{N'}}(\pi)^*=(\model_{\psi_{N'}}(\pi)^*,\model_{\psi_{N'}^{-1}}(\tilde\pi)^*,\bil_{\mira'})
\]
where $\model_{\psi_{N'}}(\pi)^*$ denotes the Whittaker model of $\pi$ with the twisted
action of $G'$ (i.e., right translation by $g^*$) and similarly for $\model_{\psi_{N'}^{-1}}(\tilde\pi)^*$.
Clearly,
\[
\Bes_{\data_{\psi_{N_M}}(\pi\otimes\pi)}^{\ell_{M_H},\dirac_e}(f)=
\Bes_{\data_{\psi_{N'}}(\pi)\otimes\data_{\psi_{N'}}(\pi)^*}^{\ell_{M_H},\dirac_e\otimes\dirac_e}(\tilde f)
\]
where $\tilde f\in\swrz(M)$ is given by $\tilde f(\sm{g_1}{}{}{g_2})=f(\sm{g_1}{}{}{g_2^*})$.
On the other hand, using Lemma \ref{lem: beselem} part \ref{part: inter} for the morphism
\[
(\id\otimes (W\mapsto W^*),\id\otimes (W\mapsto W^*)):
\data_{\psi_{N'}}(\pi)\otimes\data_{\psi_{N'}}(\pi)^*\rightarrow
\data_{\psi_{N'}}(\pi)\otimes\data_{\psi_{N'}}(\pi)^\circ
\]
we get
\[
\Bes_{\data_{\psi_{N'}}(\pi)\otimes\data_{\psi_{N'}}(\pi)^*}^{\ell_{M_H},\dirac_e\otimes\dirac_e}(\tilde f)
=\Bes_{\data_{\psi_{N'}}(\pi)\otimes\data_{\psi_{N'}}(\pi)^\circ}^{\bil_{\mira'},\dirac_e\otimes\dirac_e}(\tilde f).
\]
Finally, we use Lemma \ref{lem: beselem} part \ref{part: convo}.
\end{proof}

In order to compete the proof of Theorem \ref{thm: cmprbes}
it remains to note that by \eqref{eq: mod}, for any $f\in\swrz(G)$ and $g\in G'$ we have
\begin{align*}
\Conv(f^\sharp)(g)=&\abs{\det g}^{\frac{n+1}2}\int_{M_H}\modulus_{P_H}(m)\int_{\K_H}\int_Uf(km\crnr(g)u)\ du\ dk\ dm
\\=&\abs{\det g}^{\frac{1-n}2}\int_{M_H}\modulus_{P_H}(m)\int_{\K_H}\int_Uf(kmu\crnr(g))\ du\ dk\ dm
\\=&\abs{\det g}^{\frac{1-n}2}\int_{M_H}\modulus_{P_H}(m)\int_{\K_H}\int_{U_H\bs U}\int_{U_H}f(kmuv\crnr(g))\ du\ dv\ dk\ dm
\\=&\abs{\det g}^{\frac{1-n}2}\int_{U_H\bs U}\int_H f(hv\crnr(g))\ dh\ dv=\trns f(g).
\end{align*}

\section{Proof of Proposition \ref{prop: mainper}} \label{sec: proofmain}

In this section we prove Proposition \ref{prop: mainper}, which was the key ingredient in the proof of Theorem \ref{thm: cmprbes}.
In fact, in Corollary \ref{cor: modeltrans} below we prove a more precise statement.
Let $\pi\in\Irr_{\temp}G'$ and $\sigma=\Speh(\pi)$.
Following \cite{MR3431601} and its terminology, we will construct an explicit isomorphism (\emph{model transition})
\[
\trans_H^{(N,\psi_N)}:\model_H(\sigma)\rightarrow\model_{\psi_N}(\sigma)
\]
given by a regularized integral. The inverse map
\[
\trans_{(N,\psi_N)}^H:\model_{\psi_N}(\sigma)\rightarrow\model_H(\sigma)
\]
will also be given by an explicit (convergent) integral.

\subsection{Relative basic spherical functions} \label{sec: relbasic}
Let $\Xi^{G'}$ be the basic spherical function for $G'$ (see \cite{MR946351}).
Define $\Xi_\varpi$ to be the positive, left $U$-invariant, right $\K$-invariant function on $G$ given by
\begin{equation} \label{def: Xivarpi}
\Xi_\varpi(umk)=\modulus_P(m)^{1/2}\varpi(m)\Xi^{G'}(\proj_{H_M}(m)),\ \ u\in U, m\in M, k\in\K%
\end{equation}
where
\[
\proj_{H_M}:M\rightarrow G'
\]
is the left $M_H$-invariant function given by
\[
\proj_{H_M}(\sm{g_1}{}{}{g_2})=(g_2^*)^{-1}g_1.
\]
Note that $\Xi_\varpi$ is well-defined since $\Xi^{G'}\circ\proj_{H_M}$ is right $\K\cap M$-invariant.
Similarly define $\Xi_{\varpi^{-1}}$.
Note that by \eqref{eq: mod}, $\Xi_\varpi$ is left $(M_H,\modulus_{P_H})$-equivariant
and $\Xi_{\varpi^{-1}}$ is left $(M_H,\modulus_P\modulus_{P_H}^{-1})$-equivariant.

We define the basic spherical function $\Xi^{H\bs G}$ for $H\bs G$ by
\begin{equation} \label{def: xi}
\Xi^{H\bs G}(g)=\int_{P_H\bs H}\Xi_\varpi(hg)\ dh.
\end{equation}
Define $\reltemp(H\bs G)$ to be the space of smooth functions on $H\bs G$ that
are majorized by a constant multiple of $\Xi^{H\bs G}$, endowed with the norm
\[
\sup_{g\in H\bs G}\abs{f(g)}\Xi^{H\bs G}(g)^{-1}.
\]
(We could have defined more generally the larger space of tempered functions on $H\bs G$, but for our
purposes this is not necessary.)

Similarly, let $\reltemp(UM_H\bs G;\modulus_P\modulus_{P_H}^{-1})$ be the space of
smooth, left $U$-invariant and left $(M_H,\modulus_P\modulus_{P_H}^{-1})$-equivariant functions $f$ on $G$
that are majorized by a constant multiple of $\Xi_{\varpi^{-1}}$.
We may identify $\reltemp(UM_H\bs G;\modulus_P\modulus_{P_H}^{-1})$ with the induced space
\[
I_P(\reltemp(M_H\bs M),\varpi^{-1})
\]
where $\reltemp(M_H\bs M)$ denotes the space of smooth,
left $M_H$-invariant functions on $M$ whose precomposition with $\crnr:G'\rightarrow M$ is majorized by a constant multiple of $\Xi^{G'}$.

Finally, let $\reltempw(N_M\bs M;\psi_{N_M})$ be the space of smooth, $(N_M,\psi_{N_M})$-equivariant tempered functions $f$ on $M$
(see \cite{1812.00047}*{\S2.4}).
Let $\reltempw_{\varpi^{-1}}(N\bs G;\psi_N)$ be the space of smooth,
left $(N,\psi_N)$-equivariant functions $f$ on $G$ such that for all $g\in G$,
the function
\[
m\in M\mapsto\modulus_P^{-\frac12}(m)\varpi(m)f(mg)
\]
belongs to $\reltempw(N_M\bs M;\psi_{N_M})$.
We may identify $\reltempw_{\varpi^{-1}}(N\bs G;\psi_N)$ with
\[
I_P(\reltempw(N_M\bs M;\psi_{N_M}),\varpi^{-1}).
\]

\begin{lemma}
The map
\[
f\mapsto\int_{U_H\bs U}f(ug)\ du
\]
defines a continuous linear operator
\[
\reltemp(H\bs G)\rightarrow\reltemp(UM_H\bs G;\modulus_P\modulus_{P_H}^{-1}).
\]
\end{lemma}

\begin{proof}
We may write
\[
\Xi^{H\bs G}(g)=\int_{U_H}\Xi_\varpi(w_U^{-1} ug)\ du.
\]
Thus,
\[
\int_{U_H\bs U}\Xi^{H\bs G}(ug)\ du=\int_U\Xi_\varpi(w_U^{-1} ug)\ du,
\]
and the argument of \cite{MR3431601}*{Lemma 4.5} (using the Gindikin--Karpelevich formula) shows that
this is a constant multiple of $\Xi_{\varpi^{-1}}(g)$.
The lemma follows.
\end{proof}

Recall that the integral
\[
f\in\swrz(M_H\bs M)\mapsto\int_{N_{M_H}\bs N_M}f(um)\psi_{N_M}(u)^{-1}\ du,\ \ m\in M
\]
extends to a continuous linear operator
\[
\reltemp(M_H\bs M)\rightarrow\reltempw(N_M\bs M;\psi_{N_M})
\]
which we denote by
\[
\int_{N_{M_H}\bs N_M}^{\reg}f(um)\psi_{N_M}(u)^{-1}\ du
\]
\cite{1812.00047}*{Lemma 2.14.1}.
(For instance, it can be defined using the setting of \cite{MR3431601}*{\S2}.)
It gives rise to a continuous linear operator
\[
\reltemp(UM_H\bs G;\modulus_P\modulus_{P_H}^{-1})\rightarrow\reltempw_{\varpi^{-1}}(N\bs G;\psi_N).
\]
By the lemma above, we infer that the map
\[
f\in\swrz(H\bs G)\mapsto\int_{N_H\bs N}f(ug)\psi_N(u)^{-1}\ du
\]
extends to a continuous linear operator
\[
\reltemp(H\bs G)\rightarrow\reltempw_{\varpi^{-1}}(N\bs G;\psi_N),
\]
which we write as
\[
\int_{N_H\bs N}^{\reg}f(ug)\psi_N(u)^{-1}\ du.
\]

\subsection{From symplectic to Zelevinsky}\label{sec: mod symp to zel}
\begin{lemma}\label{lem: intM}
Let $\pi\in\Irr_{\temp}G'$ and $\sigma=\Speh(\pi)$.
Then, $\model_H(\sigma)\subseteq\reltemp(H\bs G)$.
In other words, $\transs\varphi\in \reltemp(H\bs G)$ for any
$\varphi\in I_P(\model_{\psi_{N_M}}(\pi\otimes\pi),\varpi)$.
Moreover, for any $g\in G$ we have
\[
\int_{N_H\bs N}^{\reg} \transs\varphi(ug)\psi_N(u)^{-1}\ du =(\tilde\M_\pi\varphi)(g).
\]
\end{lemma}

\begin{proof}
For any $W\in\model_{\psi_{N_M}}(\pi\otimes\pi)$, the function
\[
f(m)=\ell_{M_H}(R(m)W),\ \ m\in M
\]
is left $M_H$-invariant and $g\in G'\mapsto f(\crnr(g))$ is a matrix coefficient for $\pi$.
Thus, since $\pi$ is tempered we have
\[
\abs{f(m)}\ll_W\Xi^{G'}(\proj_{H_M}(m)), \ \ \ m\in M
\]
\cite{MR946351}.
It follows that for every $\varphi\in I_P(\model_{\psi_{N_M}}(\pi\otimes\pi),\varpi)$,
\[
\ell_{M_H}(\varphi(g))\ll_\varphi \Xi_\varpi(g),\ \ g\in G.
\]
The first part now follows from \eqref{def: transs}, \eqref{def: ellHind} and \eqref{def: xi}.

For the second part we may assume $g=e$. We have
\begin{align*}
&\int_{U_H\bs U} \transs\varphi(ug)\ du=
\int_{U_H\bs U}\int_{U_H} \ell_{M_H}(I_P(ug,\varpi)\varphi(w_U^{-1} v))\ dv\ du
\\=&\int_{U_H\bs U}\int_{U_H} \ell_{M_H}(\varphi(w_U^{-1} vug))\ dv\ du=
\int_U \ell_{M_H}(\varphi(w_U^{-1} ug))\ du=\ell_{M_H}^*(\M_\pi\varphi(g)).
\end{align*}
Hence,
\begin{align*}
&\int_{N_H\bs N}^{\reg}  \transs\varphi(u)\psi_N(u)^{-1}\ du =
\int_{N_{M_H}\bs N_M}^{\reg} \int_{U_H\bs U}\transs\varphi(uv)\ du\ \psi_{N_M}(v)^{-1}\ dv
\\=&\int_{N_{M_H}\bs N_M}^{\reg}\ell_{M_H}^*(\M_\pi \varphi(v))\psi_{N_M}(v)^{-1}\ dv
=\int_{N_{M_H}\bs N_M}^{\reg}\ell_{M_H}^*(R(v)(\M_\pi \varphi(e))) \psi_{N_M}(v)^{-1}\ dv.
\end{align*}
It follows from \cite{MR3267120}*{Lemma 4.4} that the right-hand side equals $(\M_\pi\varphi)(e)(e)$, as required.
(See also \cite{1812.00047}*{Proposition 2.14.3} and the discussion preceding it.)
\end{proof}

\begin{corollary}\label{cor: symptozel}
The map
\[
\trans_H^{(N,\psi_N)}:\model_H(\sigma)\rightarrow \model_{\psi_N}(\sigma),\ \
L\mapsto\int_{N_H\bs N}^{\reg}  L(u\cdot)\psi_N(u)^{-1}\ du
\]
is a model transition from the symplectic model of $\sigma$ to its Zelevinsky model.
Moreover,
\[
\trans_H^{(N,\psi_N)}\circ\transs=\tilde\M_\pi.
\]
\end{corollary}

\subsection{From Zelevinsky to symplectic} \label{sec: zel2sym}

In this subsection we compute the inverse of the model transition $\trans_H^{(N,\psi_N)}$,
which will also yield Proposition \ref{prop: mainper}.

We continue to use the notation introduced in \S\ref{sec: mod symp to zel}.

\begin{lemma} \label{lem: abs conv}
For any $W\in\reltempw_{\varpi^{-1}}(N\bs G;\psi_N)$ the integral
\[
\int_{N_H\bs\mira_H} W(h)\ dh=\int_{N_{H_{n-1}}\bs H_{n-1}} W(h)\ dh
\]
is absolutely convergent and defines a continuous $\mira_H$-invariant linear form on
$\reltempw_{\varpi^{-1}}(N\bs G;\psi_N)$.
\end{lemma}

\begin{proof}
Since
\[
\int_{N_H\bs\mira_H}\abs{W(h)}\ dh=
\int_{P_H\cap\mira_H\bs\mira_H}\int_{N_H\bs P_H\cap\mira_H}
 \modulus_{\mira_H}(p)\modulus_{\mira_H\cap P_H}^{-1}(p)   \abs{W(pq)}\ dp\ dq
\]
and the outer integration is compact, it is enough to show the convergence of
\[
\int_{N_H\bs P_H\cap\mira_H} \modulus_{\mira_H}(p)\modulus_{\mira_H\cap P_H}^{-1}(p)\abs{W(p)}\ dp
=\int_{N_{M_H}\bs M_H\cap\mira_H} \modulus_{\mira_H}(m)\modulus_{\mira_H\cap P_H}^{-1}(m)\abs{W(m)}\ dm.
\]
Note that $\embd$ identifies $N'\bs\mira'$ with $N_{M_H}\bs\mira_H\cap M_H$.
Moreover, a simple computation shows that
\[
\modulus_{\mira_H}(m)\modulus_{\mira_H\cap P_H}^{-1}(m) \varpi(m)^{-1} \modulus_P^{\frac12}(m)=\abs{\det g},
 \ \ \ m=\embd(g), \ g\in\mira'.
\]
Therefore, the lemma follows from the convergence of
\[
\int_{N'\bs\mira'} \abs{\det g} \abs{W'(\embd(g))}\ dg
\]
for any $W'\in\reltemp(N_{M_H}\bs N_M;\psi_{N_M})$.
In turn, this follows from \cite{1812.00047}*{Lemma 2.15.1} and the fact that $\abs{\det}$ is bounded
on the support of $W'\circ\embd$.
\end{proof}


\label{sec: modelfin}

\begin{proposition} \label{prop: modeltrans}
For any $L\in\reltemp(H\bs G)$ we have
\[
\int_{N_H\bs\mira_H}\big(\int_{N_H \bs N}^{\reg}L(uxg)\psi_N^{-1}(u)\ du\big) \ dx=L(g),\ g\in G.
\]
\end{proposition}

\begin{proof}

We may assume that $g=e$.
By continuity, we may also assume that $L\in\swrz(H\bs G)$.
Denote by $\overline{X}$ the image under transpose of a subgroup $X$ of $G$.
Let $B_{H_{n-1}}$ be the Borel subgroup of upper triangular matrices in $H_{n-1}$.
For any left $N_H$-invariant function $f$ on $\mira_H$ we may write
\[
\int_{N_H\bs\mira_H}f(h)\ dh=\int_{\overline{B_{H_{n-1}}}}f(b)\ db
\]
provided that the left-hand side converges.
We decompose the integral over $\overline{B_{H_{n-1}}}$ as follows.
Write $T_H=\prod_{k=1}^n\tilde T_k$ where $\tilde T_k$ is the image of the co-root
\[
\beta_k^\vee(a)=\embd(\left(\begin{smallmatrix}I_{n-k}&&\\&a&\\&&I_{k-1}\end{smallmatrix}\right))^{-1}=
\left(\begin{smallmatrix}I_{n-k}&&&&\\&a&&&\\&&I_{2(k-1)}&&\\&&&a^{-1}&\\&&&&I_{n-k}\end{smallmatrix}\right),\ \ a\in F^*.
\]
Let $V_k$, $k=1,\dots,n$ be the subgroup of $N$ consisting of the elements
$\left(\begin{smallmatrix}I_{n-k}&&\\&u&\\&&I_{n-k}\end{smallmatrix}\right)$
where the middle $2(k-1)\times 2(k-1)$-block of $u$ is $I_{2(k-1)}$.
(Thus, $V_k$ is a Heisenberg group of dimension $4k-3$ while $(V_k)_H$ is a Heisenberg group of dimension $2k-1$.)
Then,
\[
\int_{\overline{B_{H_{n-1}}}}f(b)\ db=
\int_{\tilde T_{n-1}}\int_{(\overline{V}_{n-1})_H}\cdots \int_{\tilde T_1}\int_{(\overline{V}_1)_H}
f(\bar v_1t_1\cdots \bar v_{n-1}t_{n-1})\ d\bar v_1\ dt_1\cdots \ d\bar v_{n-1}\ dt_{n-1}.
\]
For $k=1,\dots,n$ let $\Nu_k=V_{k+1}\cdots V_n$ be the unipotent radical of the parabolic subgroup of $G$
of type $(\overbrace{1,\dots,1}^{n-k},2k,\overbrace{1,\dots,1}^{n-k})$.
Consider the following integrals
\[
J_k(L;g)=\int_{(\Nu_k)_H\bs \Nu_k}L(ug)\psi_N(u)^{-1}\ du, \ \ \ L\in\swrz(H\bs G).
\]

Observe that
\begin{equation}\label{eq JTk1}
J_{k+1}(L;tg)=J_{k+1}(L;g),\ \ \ g\in G, \ t\in \tilde T_k
\end{equation}
since $L$ is left $\tilde T_k$-invariant,
$\tilde T_k$ stabilizes $\psi_N\rest_{\Nu_{k+1}}$ and conjugation by $\tilde T_k$ preserves the Haar measures
on $\Nu_{k+1}$ and $(\Nu_{k+1})_H$.

Also, for $k=1,\dots,n-1$ let $R_k\subset V_{k+1}$ be the one-parameter unipotent group corresponding to the simple root
$\alpha_{n-k}$ in the usual enumeration.
Define
\[
J_k'(L;g)=\int_{R_k}J_{k+1}(L;ug)\psi_N(u)^{-1}\ du, \ \ \ L\in\swrz(H\bs G).
\]

Note that
\[
J_1(L;g)=\int_{(\Nu_1)_H\bs\Nu_1} L(ug)\psi_N(u)^{-1}\ du=\int_{N_H\bs N} L(ug)\psi_N(u)^{-1}\ du,
\]
while $\Nu_n=1$, so that $J_n(L;g)=L(g)$.
The proposition therefore follows from the following two statements, which will be proved below.
\begin{subequations}
\begin{equation} \label{eq: Jkint}
\int_{(\overline{V}_k)_H} J_k(L;\bar vg)\ d\bar v=J_k'(L;g),
\end{equation}
\begin{equation} \label{eq: Jk'int}
\int_{\tilde T_k} J_k'(L;t)\ dt=J_{k+1}(L;e).
\end{equation}
\end{subequations}

Let $C=(\overline{V}_k)_H$ and $D=V_{k+1}\cap (N_M^\der \ltimes U)$
where $N_M^\der$ is the derived group of $N_M$.
It is easy to verify that $[C,D]\subseteq V_{k+1}$ and the map $c\mapsto \psi_N({[c,\cdot}])$ defines a
homeomorphism from $C$ to the Pontryagin dual of $D_H\bs D$.
Therefore the function $f(g)=J_k(L;g)$ satisfies
\[
f(cd)=f([c,d]dc)=\psi_N([c,d])\psi_N(d)f(c).
\]
It follows from \cite{MR3431601}*{Lemma 2.9} that $\bar v\mapsto f(\bar v g)$ is supported on a compact subset of $C$.
The desired equality \eqref{eq: Jkint} is
\begin{align*}
&\int_{(\overline{V}_k)_H}\left(\int_{(\Nu_k)_H\bs \Nu_k} L (u\bar vg)\psi_N(u)^{-1}\ du\right) d\bar v
\\=&\int_{R_k}\int_{(\Nu_{k+1})_H\bs \Nu_{k+1}} L(uvg)\psi_N(u)^{-1}\ du\ \psi_N(v)^{-1}\ dv
\end{align*}
where the left-hand side converges as an iterated integral and the right-hand side converges as a double integral.
This follows from \cite{MR3431601}*{Lemma A.1} (which amounts to Fourier inversion) with $G_0=G$, $H_0=H$, $f=L(\cdot g)$,
$A=\Nu_k$, $B=R_k(V_{k+1})_H\ltimes \Nu_{k+1}$, $C$, $D$ as above and
$\Psi(u\bar u)=\psi_N(u)^{-1}$ for $u\in N$ and $\bar u\in \overline{N}$.
It is straightforward to check that the conditions of [ibid.] are satisfied.

Let $\lambda_k:F\rightarrow R_k$ be the isomorphism whose inverse is the $(n-k,n-k+1)$-coordinate
and let $h\in\swrz(F)$ be defined by
\begin{equation} \label{def: h}
h(x)=J_{k+1}(L;\lambda_k(x)).
\end{equation}
Then by \eqref{eq JTk1} and Fourier inversion we have
\begin{align*}
&\int_{\tilde T_k} J_k'(L;t)\ dt=
\int_{F^*}\big(\int_F J_{k+1}(L;\lambda_k(x)\beta_k^\vee(a))\psi(x)^{-1}\ dx\big)\ d^*a\\
=&\int_{F^*}\big(\int_F h(a^{-1}x)\psi(x)^{-1}\ dx\big)\ d^*a=
\int_F\big(\int_F h(x)\psi(a x)^{-1}\ dx\big)\ da=h(0)=J_{k+1}(L;e),
\end{align*}
hence \eqref{eq: Jk'int}.
\end{proof}

Combined with Lemma \ref{lem: intM}, Corollary \ref{cor: symptozel} and Remark \ref{rem: symp} we can conclude

\begin{corollary} \label{cor: modeltrans}
Let $\pi\in\Irr_{\temp}(G')$ and $\sigma=\Speh(\pi)$. Then,
\begin{enumerate}
\item For any $\varphi\in I_P(\model_{\psi_{N_M}}(\pi\otimes\pi),\varpi)$ we have
\[
\int_{N_H\bs\mira_H}\big(\int_{N_H \bs N}^{\reg} \transs\varphi(uxg)\psi_N^{-1}(u)\ du\big) \ dx=\transs\varphi(g),\ g\in G
\]
where $\transs\varphi$ is defined in \eqref{def: transs}.
\item The functional
\[
\ell_H(W)=\int_{N_H\bs\mira_H}W(h)\ dh
\]
on $\model_{\psi_N}(\sigma)$ is $(\tilde H,\omega_\pi\circ\tilde\lambda)$-equivariant,
and in particular, $H$-invariant.
Moreover, the model transition
\[
\trans_{(N,\psi_N)}^H:\model_{\psi_N}(\sigma)\rightarrow \model_H(\sigma),\ \
W\mapsto\int_{N_H\bs\mira_H}W(h\cdot)\ dh
\]
is the inverse of $\trans_H^{(N,\psi_N)}$.
\item $\transs=\trans_{(N,\psi_N)}^H\circ\tilde\M_\pi$. Thus, Proposition \ref{prop: mainper} holds.
\end{enumerate}
\end{corollary}

\section{Completion of proof}

In this section we prove Theorem \ref{thm: main}.

\subsection{}
The first step is to take $\varphi_i=\int_Hf_i(h\cdot)\ dh$, $i=1,2$, for which the putative relation \eqref{eq: inner} becomes
\begin{equation} \label{eq: intH}
\int_Hf(h)\ dh=\int_{\Irr_{\temp}(G')}\Bes_{\data_{\psi_N}(\Speh(\pi))}^{\ell_H,\ell_H}(f^\vee)\ d\mu_{\pl}(\pi)
\end{equation}
where $f=\overline{f_2}*f_1^\vee$.
Thus, we need to show \eqref{eq: intH} for any $f\in\swrz(G)$.

Next, we use an inversion formula for the left-hand side.
As in \cite{1812.00047}*{\S2.14} define
\[
W^{G'}_{f'}(g_1,g_2)=\int_{N'}f'(g_1^{-1}ug_2)\psi_{N'}(u)^{-1}\ du,\ \ g_1,g_2\in G'
\]
for any $f'\in\swrz(G')$. Similarly, let
\[
W_f(g_1,g_2)=\int_Nf(g_1^{-1}ug_2)\psi_N(u)^{-1}\ du=\int_{N_M}\int_Uf(g_1^{-1}vug_2)\ du\ \psi_{N_M}(v)^{-1}\ dv, \ \ g_1,g_2\in G.
\]
Thus, $W_f(u_1g_1,u_2g_2)=\psi_N(u_1^{-1}u_2)W_f(g_1,g_2)$ for any $u_1,u_2\in N$, $g_1,g_2\in G$.

The first part of the following lemma is an analogue of \cite{1812.00047}*{Proposition 4.3.1}.

\begin{lemma}[local unfolding] \label{lem: inversion}
For any $f\in\swrz(G)$ we have
\begin{equation} \label{eq: BPinv}
\int_Hf(h)\ dh=\int_{N_H\bs \mira_H}\int_{N_H\bs H}W_f(h,q)\ dh\ dq
\end{equation}
where the right-hand side converges as an iterated integral. Moreover,
\[
\int_{N_H\bs H}W_f(h,g)\ dh=W_{\trns (R(g)f)}^{G'}(e,e), \ \ \ g\in G.
\]
(see \eqref{def: trns}).
\end{lemma}

\begin{remark}
As in \cite{1812.00047}, it can be shown that the right-hand side of \eqref{eq: BPinv} converges as a double integral.
We will not give details since we will not need to use this fact.
\end{remark}

\begin{proof}
Let $\varphi=\int_Hf(h\cdot)\ dh$.
Then,
\begin{align*}
&\int_{N_H\bs H}W_f(h,g)\ dh=\int_{N_H\bs H}\int_Nf(h^{-1}ug)\psi_N(u)^{-1}\ du\ dh
\\=&\int_{N_H\bs H}\int_{N_H\bs N}\int_{N_H}f(h^{-1}vug)\psi_N(u)^{-1}\ dv\ du\ dh
\\=&\int_{N_H\bs N}\int_{N_H\bs H}\int_{N_H}f(h^{-1}vug)\psi_N(u)^{-1}\ dv\ dh\ du
\\=&\int_{N_H\bs N}\int_Hf(h^{-1}ug)\psi_N(u)^{-1}\ dh\ du=
\int_{N_H\bs N}\varphi(ug)\psi_N(u)^{-1}\ du.
\end{align*}
This is justified since $g\mapsto\int_H\abs{f(h^{-1}g)}\ dh$ is compactly supported in $H\bs G$
and $N_H\bs N$ is closed in $H\bs G$.
Therefore, the identity \eqref{eq: BPinv} becomes
\[
\int_{N_H\bs \mira_H}\big(\int_{N_H\bs N}\varphi(uh)\psi_N(u)^{-1}\ du\big)\ dh=\varphi(e),
\]
which follows from Proposition \ref{prop: modeltrans}.

For the second part, we may assume that $g=e$. As above, we write
\begin{align*}
&\int_{N_H\bs H}W_f(h,e)\ dh=\int_{N_H\bs N}\int_Hf(hu)\psi_N(u)^{-1}\ dh\ du
\\=&\int_{N_{M_H}\bs N_M}\int_{U_H\bs U}\int_Hf(huv)\psi_{N_M}(v)^{-1}\ dh\ du\ dv
=\int_{N'}\trns f(v)\psi_{N'}(v)^{-1}\ dv
\\=&W_{\trns f}^{G'}(e,e)
\end{align*}
as required.
\end{proof}

Note that the second part of Lemma \ref{lem: inversion} is elementary, in contrast to the much more involved
analogous step in \cite{1812.00047}.

\subsection{}
Next, we recall the Whittaker spectral expansion for $G'$, explicated in Proposition 2.14.2 and (2.14.3) of \cite{1812.00047}
(see also \cite{1812.00047}*{\S2.8}).

\begin{proposition} \label{prop: specexpGLn}
For any $f\in \swrz(G')$
\[
W_f^{G'}(e,e)=\int_{\Irr_{\temp}(G')}\Bes_{\data_{\psi_{N'}}(\pi)}^{\dirac_e,\dirac_e}(f^\vee)\ d\mu_{\pl}(\pi)
=\int_{\Irr_{\temp}(G')}\Bes_{\data_{\psi_{N'}^{-1}}(\pi)}^{\dirac_e,\dirac_e}(f)\ d\mu_{\pl}(\pi)
\]
(see Lemma \ref{lem: beselem}). Moreover, the integral
\[
\int_{N'\bs\mira'^*}\int_{\pi\in\Irr_{\temp}(G')}\abs{\Bes_{\data_{\psi_{N'}^{-1}}(\pi)}^{\dirac_m,\dirac_m}(f)}\ d\mu_{\pl}(\pi)\ dm
\]
converges.
\end{proposition}

Combining Lemma \ref{lem: inversion} and Proposition \ref{prop: specexpGLn} we get
\[
\int_Hf(h)\ dh=
\int_{N_H\bs \mira_H}\int_{\pi\in\Irr_{\temp}(G')}\Bes_{\data_{\psi_{N'}^{-1}}(\pi)}^{\dirac_e,\dirac_e}
(\trns(R(q)f))\ d\mu_{\pl}(\pi)\ dq.
\]
We claim that the double integral on the right-hand side converges. Indeed, we can write it as
\[
\int_{P_H\cap \mira_H\bs \mira_H}\int_{N_{M_H}\bs M_H\cap \mira_H}\int_{\pi\in\Irr_{\temp}(G')}
\modulus_{P_H\cap \mira_H}^{-1}(m)
\Bes_{\data_{\psi_{N'}^{-1}}(\pi)}^{\dirac_e,\dirac_e}(\trns(R(m)R(h)f))\ d\mu_{\pl}(\pi)\ dm\ dh.
\]
Since $\modulus_{P_H\cap \mira_H}(\embd(m))=\abs{\det m}^{-n}$ for any $m\in\mira'$ and
\[
\trns(R(\embd(m))f)(g)=\abs{\det m}^{1-n}\trns f((m^*)^{-1}gm^*),\ \ m,g\in G',
\]
we get by Lemma \ref{lem: beselem}
\[
\int_{\K\cap\mira_H}\int_{N'\bs\mira'}\int_{\pi\in\Irr_{\temp}(G')}\abs{\det m}
\Bes_{\data_{\psi_{N'}^{-1}}(\pi)}^{\dirac_{m^*},\dirac_{m^*}}(\trns(R(k)f))\ d\mu_{\pl}(\pi)\ dm\ dk.
\]
Therefore, the claim follows from the second part of Proposition \ref{prop: specexpGLn} (applied to
$\trns(R(k)f)\in\swrz(G')$) since $\abs{\det m}$ is bounded on the support of the function
\[
m\mapsto\Bes_{\data_{\psi_{N'}^{-1}}(\pi)}^{\dirac_{m^*},\dirac_{m^*}}(\trns(R(k)f)),\ \ m\in\mira',
\]
which in turn follows from \eqref{eq: besOB} and the fact that for any compact open subgroup $K'$ of
$G'$ and any $W'\in\model_{\psi_{N'}}(\pi)^{K'}$, $\abs{\det m}$ is bounded on the support of
$m\mapsto W'(m^*)$, $m\in\mira'$, in terms of $K'$ only.

Interchanging the order of integration we get
\[
\int_Hf(h)\ dh=
\int_{\pi\in\Irr_{\temp}(G')}\int_{N_H\bs \mira_H}\Bes_{\data_{\psi_{N'}^{-1}}(\pi)}^{\dirac_e,\dirac_e}
(\trns(R(q)f))\ dq\ d\mu_{\pl}(\pi)
\]
which by Theorem \ref{thm: cmprbes} and Lemma \ref{lem: beselem} is equal to
\begin{align*}
&\int_{\pi\in\Irr_{\temp}(G')}\int_{N_H\bs \mira_H}
\Bes_{\data_{\psi_N^{-1}}(\Speh(\pi))}^{\ell_H,\dirac_e}(R(q)f)\ dq\ d\mu_{\pl}(\pi)
\\=&\int_{\pi\in\Irr_{\temp}(G')}\int_{N_H\bs \mira_H}
\Bes_{\data_{\psi_N^{-1}}(\Speh(\pi))}^{\ell_H,\dirac_q}(f)\ dq\ d\mu_{\pl}(\pi)
\\=&\int_{\pi\in\Irr_{\temp}(G')}\Bes_{\data_{\psi_N^{-1}}(\Speh(\pi))}^{\ell_H,\ell_H}(f)\ d\mu_{\pl}(\pi)
=\int_{\pi\in\Irr_{\temp}(G')}\Bes_{\data_{\psi_N}(\Speh(\pi))}^{\ell_H,\ell_H}(f^\vee)\ d\mu_{\pl}(\pi).
\end{align*}
All in all, this gives the relation \eqref{eq: intH}, which is equivalent to Theorem \ref{thm: main}.

\subsection{Symplectic similitude group}
From Theorem \ref{thm: main} we can immediately deduce a variant for $\tilde H=\GSp_n$.
Fix a character $\chi$ of $F^*$.
For any $\varphi\in\swrz(H\bs G)$ let
\[
\tilde\varphi(x)=\int_{H\bs\tilde H}\varphi(tx)\chi(\tilde\lambda(t))\ dt.
\]
Then, $\tilde\varphi\in\swrz(\tilde H\bs G;\chi^{-1})$ and for any $\varphi_1,\varphi_2\in\swrz(H\bs G)$ we have
\[
(\tilde\varphi_1,\tilde\varphi_2)_{L^2(\tilde H\bs G;\chi^{-1})}=
\int_{H\bs\tilde H}(\varphi_1(t\cdot),\varphi_2)_{L^2(H\bs G)}\chi(\tilde\lambda(t))\ dt.
\]
On the other hand, it follows from Corollary \ref{cor: modeltrans} that
for any $\pi\in\Irr_{\temp}(G')$, $\varphi_1,\varphi_2\in\swrz(H\bs G)$ and $t\in\tilde H$ we have
\[
(\varphi_1(t\cdot),\varphi_2)_\sigma=\omega_\pi(\tilde\lambda(t))^{-1}(\varphi_1,\varphi_2)_\sigma
\]
where $\sigma=\Speh(\pi)$. Therefore, by Theorem \ref{thm: main}
\[
(\tilde\varphi_1,\tilde\varphi_2)_{L^2(\tilde H\bs G;\chi^{-1})}=
\int_{H\bs\tilde H}\big(\int_{\Irr_{\temp}(G')}(\varphi_1,\varphi_2)_{\Speh(\pi)}\omega_\pi^{-1}(\tilde\lambda(t))
\ d\mu_{\pl}(\pi)\big)\chi(\tilde\lambda(t))\ dt
\]
where the integral converges as an iterated integral.
By \eqref{eq: defmuchi}, this is equal to
\[
\int_{\Irr_{\temp}^\chi(G')}(\varphi_1,\varphi_2)_{\Speh(\pi)}\ d\mu_{\pl}^\chi(\pi).
\]
This concludes the proof of Theorem \ref{thm: gsp}.

\subsubsection*{Acknowledgement}
The paper owes a great deal to Herv\'e Jacquet who kindly shared with us his input on the problem some time ago.
We thank him for his encouragement and inspiration along the years.
We also owe an intellectual debt to Rapha\"{e}l Beuzart-Plessis for his paper \cite{1812.00047}
that influenced the present paper decisively.
We also thank him for suggesting a simplification in the argument of \S\ref{sec: proofmain}.
We thank Dmitry Gourevitch for providing the appendix below explaining how to remove
the characteristic $0$ assumption in \cite{1806.10528}.
Finally, we thank Zhengyu Mao for useful discussions.

\appendix

\newcommand{\Rep}{\operatorname{Rep}}
\newcommand{\diag}{\operatorname{diag}}
\newcommand{\Q}{{\mathbb Q}}
\newcommand{\R}{{\mathbb R}}
\newcommand{\fg}{{\mathfrak{g}}}
\newcommand{\fn}{{\mathfrak{n}}}
\newcommand{\fu}{{\mathfrak{u}}}
\newcommand{\gl}{{\mathfrak{gl}}}
\newcommand{\cW}{\mathcal{W}}
\newcommand{\onto}{{\twoheadrightarrow}}
\newcommand{\depth}{\operatorname{depth}}

\section{Generalized Whittaker and Zelevinsky models for the general linear group, by Dmitry Gourevitch\texorpdfstring{\footnote{Faculty of Mathematics
and Computer Science, Weizmann Institute of Science, POB 26, Rehovot 76100, Israel
\\{\email{dmitry.gourevitch@weizmann.ac.il}}}}{}}

The goal of this appendix is to recall some results about various models for irreducible representations
of general linear groups over a local non-Archimedean field $F$.
The principal model is the one considered by Zelevinsky in \cite{MR584084} under the name ``degenerate Whittaker model''.
Subsequently, a vast generalization of these models was considered by M\oe glin--Waldspurger \cite{MR913667}
(for any reductive group where $F$ is of characteristic $0$)
who relate the dimension of these models to coefficients in the Harish--Chandra germ expansion of the character of the representation.
More recently, a different point of view was taken in \cite{MR3705224} where different models of the same representation are compared directly.
This is particularly nice in the case of the general linear group, for which there is no restriction on the characteristic of $F$.
The situation in the Archimedean case will also be explained.

\subsection{}
Let $F$ be a non-Archimedean local field of arbitrary characteristic.
Let $n$ be a positive integer and define $G:=\GL_n(F)$ and  $\fg:=\gl_n(F)$.
Let $\lambda$ be a partition of $n$, i.e., a finite sequence of integers
$\lambda_1\geq \lambda_2\geq \dots,\geq \lambda_l>0$ with $\sum_{i=1}^l\lambda_i=n$.
In this appendix we establish isomorphisms between several degenerate Whittaker models corresponding to $\lambda$.
To define these models we will need some notation.

Let $J_\lambda:=\{n-\sum_{i=1}^j\lambda_i\, \vert \, 1\le j<l\}$ and $J'_\lambda:=\{1,\dots ,n-1\}\setminus J_\lambda$.\\
Define $f_\lambda\in\fg$ by $f_\lambda=\sum_{i\in J'_\lambda}E_{i+1,i}$, where $E_{ij}$ denote the elementary matrices.

Let $\fn\subset \fg$ be the maximal nilpotent Lie subalgebra consisting of strictly upper triangular matrices.
Define a subalgebra $\fn_\lambda\subset \fn$ by
\[
\fn_\lambda=\{ A\in \fn\, \vert \, A_{i,j+1}=0 \text{ for every }j\in J_\lambda \text{ and every }i\}.
\]
Define $h_\lambda\in \gl_n(\Z)$ to be the diagonal matrix
\[
h_\lambda=\diag(\lambda_l-1,\lambda_l-3,\dots,1-\lambda_l,\dots,\lambda_1-1,\dots,1-\lambda_1).
\]
Note that there exists $e_\lambda\in \gl_n(\Q)$ such that $(e_\lambda,h_\lambda,f_\lambda)$ is an $\mathfrak{sl}_2$-triple.
Let $\gl_n(\Z)(i)$ denote the $i$-eigenspace of the adjoint action of $h_\lambda$, and let $\fg(i)$ be the corresponding subspace
of $\fg$ under the isomorphism $\fg\cong \gl_n(\Z)\otimes_{\Z}F$.
Define a nilpotent subalgebra $\fu_\lambda\subset \fg$ by
\[
\fu_\lambda=(\fg(1)\cap \fn)\oplus \bigoplus_{i\geq 2}\fg(i).
\]
In fact, $\fu_\lambda$ is the nilradical of a (non-standard) parabolic subgroup corresponding to the partition conjugate to  $\lambda$.

Let $N,N_\lambda,U_\lambda\subset G$ be the unipotent subgroups with Lie subalgebras $\fn,\fn_\lambda,\fu_\lambda\subset \fg$ given by
\[
N=\{\Id +X\,\vert \, X\in \fn\}, \quad N_\lambda=\{\Id +X\,\vert \,
X\in \fn_\lambda\}, \quad \text{and }\,U_\lambda=\{\Id +X\,\vert \, X\in \fu_\lambda\}.
\]
Fix a continuous, non-trivial, additive character $\psi$ of $F$.
For each of the groups $R=N,N_\lambda,U_\lambda$ define a character $\psi_{R,\lambda}$ on $R$ by
\[
\psi_{R,\lambda}(\Id+X)=\psi(\tr(f_\lambda X))
\]
and consider
\[
\cW_{\psi_{R,\lambda}}:=\ind_R^G\psi_{R,\lambda}
\]
where $\ind$ denotes compact induction.
The dual space of $\cW_{\psi_{R,\lambda}}$ is $\Ind_R^G\psi_{R,\lambda}^{-1}$.

\begin{theorem}[\cite{MR3705224}*{Theorem F}] \label{thm:RE}
We have an isomorphism $\cW_{\psi_{N_\lambda,\lambda}}\simeq \cW_{\psi_{U_\lambda,\lambda}}$.
\end{theorem}
This theorem is proven using the root exchange technique.
In \cite{MR3705224} it was assumed that $F$ is of characteristic $0$, but with the above definitions, the proof works for
positive characteristic as well.

\subsection{}
Let $\Rep^{\infty}(G)$ denote the category of smooth representations of $G$, and $\Irr(G)$ denote the set of
isomorphism classes of irreducible smooth representations of $G$.
For any $\pi \in\Rep^{\infty}(G)$ and each of the groups $R=N,N_\lambda,U_\lambda$, denote by $\pi_{\psi_{R,\lambda}}$
the spaces of $(R,\psi_{R,\lambda})$-coinvariants of $\pi$.
By Frobenius reciprocity for compact induction (\cite{MR0425030}*{Proposition 2.29}), we have
\begin{equation}\label{=Frob}
\pi_{\psi_{R,\lambda}}\cong (\cW_{\psi_{R,\lambda}}\otimes \pi)_{G},\ \ R=N,N_\lambda,U_\lambda
\end{equation}
where the subscript $G$ denotes coinvariants under the diagonal action of $G$ on the tensor product.
The dual space of $\pi_{\psi_{R,\lambda}}$ is canonically isomorphic to $\Hom(\pi,\Ind_R^G\psi_{R,\lambda}^{-1})$.

Since $N_\lambda$ is a subgroup of $N$, we have the natural projection
\[
\pi_{\psi_{N_\lambda,\lambda}}\onto \pi_{\psi_{N,\lambda}}.
\]
Under certain conditions, this map is an isomorphism.
In order to formulate and prove this statement more precisely we will express both spaces in terms of Bernstein--Zelevinsky derivatives
and ``pre-derivatives", that we will now define, following \cites{MR0425030,MR0579172,MR3319629}.

For any $k\leq n$ let $G_k:=\GL_k(F)$, and consider it as a subgroup of $G$ embedded in the upper left corner.
The definition of derivative is based on the so-called ``mirabolic'' subgroup $P_{n}$ of $G_{n}$ consisting of
matrices with last row $(0,\dots,0,1)$. The Levi decomposition of $P_n$ is $P_n = G_{n-1} \ltimes V_n$ where
the unipotent radical $V_n$ is an $\left( n-1\right) $-dimensional linear space and the reductive part is $G_{n-1}$.
If $n>1$, the group $G_{n-1}$ has two orbits on $V_{n}$ and hence also on the Pontryagin dual $V_n^*$ of $V_n$:
the closed one consisting of just $0$ itself, and its complement.
Let $\psi_{V_n}$ be the non-trivial character of $V_{n}$ given by
\[
\psi_{V_n}(v):=\psi(v_{n-1,n}).
\]
Then, the stabilizer of $\psi_{V_n}$ in $G_{n-1}$ is $P_{n-1}$.
Following \cite{MR0425030}*{\S 5.11}, for any $n>1$ define functors
\[
\Phi^-:\Rep^{\infty}(P_n)\rightarrow\Rep^\infty(P_{n-1}),\ \ \Psi^-:\Rep^\infty(P_n)\rightarrow\Rep^\infty(G_{n-1})
\]
by\footnote{Note that in \cite{MR0579172} the definition of these functors differs by a twist by the character $|\det|^{1/2}$.}
\[
\Phi^-(\pi):=\pi_{\psi_{V_n}},\quad \Psi^-(\pi):=\pi_{V_n}.
\]
For consistency, we also write $\Phi^-(\pi)=\Psi^-(\pi)=\pi$ if $n=1$ (in which case the groups
$P_n=G_{n-1}=P_{n-1}$ are trivial).
Denote also by $\pi|_{G_{n-1}}$ the restriction.

We then define functors
\[
D^k,E^k:\Rep^\infty(G_n)\rightarrow\Rep^\infty(G_{n-k}),\ \ 0<k\le n
\]
by
\[
D^k(\pi):=\Psi^-((\Phi^-)^{k-1}(\pi|_{P_{n}})), \quad E^k(\pi):=((\Phi^-)^{k-1}(\pi|_{P_{n}}))|_{G_{n-k}}.
\]
The representation $D^k(\pi)$ is called \emph{the $k$-th derivative of $\pi$},
and $E^k(\pi)$ is called \emph{the $k$-th pre-derivative of $\pi$}.

Observe that from the definitions, for any $\pi\in \Rep^{\infty}(G)$ and any partition $\lambda$ we have
\begin{equation} \label{eq: DEW}
\pi_{\psi_{N,\lambda}}\cong D^{\lambda_l}(D^{\lambda_{l-1}}(\dots(D^{\lambda_1}(\pi)))\quad\text{and} \quad
\pi_{\psi_{N_\lambda,\lambda}}\cong E^{\lambda_l}(E^{\lambda_{l-1}}(\dots(E^{\lambda_1}(\pi))).
\end{equation}

The basic idea behind the theory of derivatives, developed in \cite{MR0579172},
is to decompose any smooth representation $\pi$ of $P_n$ into irreducible representations of $V_n$,
thus obtaining a $G_{n-1}$-equivariant sheaf on $V_n^*$.
Then. $\Psi^-(\pi)$ is the fiber at zero (the closed orbit), and $\Phi^-(\pi)$ is the fiber at a point in the open orbit.
In particular, if $\Phi^-(\pi)$ is zero, then $V_n$ acts trivially on $\pi$, and thus $\Psi^-(\pi)$ is $\pi$.
This statement is \cite{MR0579172}*{Remark 3.3(b) and Proposition 3.2 (d,e)}.
It implies the following result.

\begin{proposition}[{}]\label{prop:DE}
Let $\pi\in \Rep^{\infty}(G)$ and $k\leq n$. Suppose that $D^{i}(\pi)=0$ for every $i>k$.
Then the natural projection $E^k(\pi)\onto D^k(\pi)$ is an isomorphism.
\end{proposition}

If $\pi\ne0$ then the maximal index $k\le n$ such that $D^k(\pi)\neq 0$ is called the \emph{depth} of $\pi$, and denoted $\depth(\pi)$.
The representation $D^{\depth(\pi)}(\pi)$ is called the \emph{highest derivative} of $\pi$.

\subsection{}
The theory of derivatives lead to the Zelevinsky classification of $\Irr(G)$ \cite{MR584084}.
Along the way, Zelevinsky proved that the highest derivative of any irreducible representation is irreducible:

\begin{proposition}[\cite{MR584084}*{Theorem 8.1}]\label{prop:DIrr}
Let $\pi\in \Irr(G)$ and $d:=\depth(\pi)$. Then $D^d(\pi)\in \Irr(G_{n-d})$.
\end{proposition}

\begin{definition}
Let $\pi\in \Irr(G)$. Let $d_1=\depth(\pi)$, $\pi_1=\pi$ and define recursively $\pi_i=D^{d_{i-1}}(\pi_{i-1})$
and $d_i=\depth(\pi_i)$, $i>1$ until $\sum_{i=1}^sd_i=n$. Then $dp(\pi)=(d_1,\dots,d_s)$
is the \emph{depth partition} of $\pi$.
\end{definition}
This partition is described combinatorially in terms of the Zelevinsky classification of $\pi$ \cite{MR584084}*{\S8}.
From this description it follows that this is indeed a partition, i.e., $d_1\geq d_2\geq \dots\geq d_s$.

\subsection{}
Summarizing all the above we obtain the following result.

\begin{proposition}\label{prop: mainapp}
Let $\pi\in \Irr(G)$ and let $\lambda$ be the depth partition of $\pi$. Then,
\[
\pi_{\psi_{U_\lambda,\lambda}}\simeq \pi_{\psi_{N_\lambda,\lambda}}\cong\pi_{\psi_{N,\lambda}},
\]
and the three spaces are one-dimensional.
\end{proposition}

\begin{proof}
The first isomorphism follows from  \eqref{=Frob} and Theorem \ref{thm:RE}.
The second one follows from \eqref{eq: DEW} and Proposition \ref{prop:DE}.
Finally, by \eqref{eq: DEW} and using Proposition \ref{prop:DIrr} repeatedly, we obtain that $\pi_{\psi_{N,\lambda}}$
is an irreducible representation of $G_0$, i.e., a one-dimensional vector space (cf. \cite{MR584084}*{Corollary 8.3}).
\end{proof}

\subsection{}
Let us end with a few remarks.
The realization of $\pi$ in $\Ind_N^G\psi_{N,\lambda}$ is the Zelevinsky model of $\pi$ considered in \cite{MR584084}
(where it was called the degenerate Whittaker model).
The realization of $\pi$ in $\Ind_{U_\lambda}^G\psi_{U_\lambda,\lambda}$ can be called the generalized Whittaker model of $\pi$.
In the case $\lambda_1=\dots=\lambda_l$ it was considered in \cite{1806.10528} under the terminology (generalized) Shalika model.
In this case, $U_\lambda$ is conjugate to the unipotent radical $V$ of the standard parabolic subgroup $P=LV\subset G$
with Levi subgroup $L=\GL_l(F)\times \dots\times\GL_{l}(F)$.
Under this conjugation, $\psi_{U_\lambda,\lambda}$ becomes the generic character $\psi_V$ of $V$ given by the value of $\psi$
on the sum of traces of super-diagonal blocks.
The stabilizer $M$ of $\psi_V$ in $L$ is the diagonal copy of $\GL_l(F)$.
Then $M$ acts on $\pi_{\psi_V}$ for any $\pi\in\Irr(G)$, and since $ \pi_{\psi_V}\cong \pi_{\psi_{U_\lambda,\lambda}}$ by conjugation,
it also acts on $\pi_{\psi_{U_\lambda,\lambda}}$.
For Speh representations, a unitary structure on the Zelevinsky model and the generalized Shalika model was given explicitly in \cite{1806.10528}.
Note that the assumption that $F$ is of characteristic $0$, made in \cite{1806.10528}, is redundant thanks to Proposition \ref{prop: mainapp}.

Consider now the Archimedean case.
We work with the category of nuclear, smooth, Fr\'echet representations of moderate growth.
The relations \eqref{=Frob}, \eqref{eq: DEW} and Theorem \ref{thm:RE} still hold  with the same proofs (see \cite{MR3705224}*{\S 4}),
provided that we take the completed tensor product and replace compact induction by Schwartz induction (see \cite{MR1100992}*{\S2}).
Unfortunately, Proposition \ref{prop:DE} does not hold in general.
For instance, if $n=2$, $k=1$ and $\pi$ is an irreducible finite-dimensional representation,
then $E^{k}(\pi)=\pi$ while $D^k(\pi)=\pi_N$ is one-dimensional.
We do not know whether Proposition \ref{prop:DIrr} holds in general.
However, Propositions \ref{prop:DE} and \ref{prop:DIrr}, and hence also \ref{prop: mainapp}, are known to hold
if $\pi$ is \emph{unitarizable} (and irreducible) -- see \cite{MR3319629}*{Theorem 4.3.1}.

The unitary structure on models of Speh representations considered in \cite{1806.10528}
is also valid in the Archimedean case, with the same proof.
The invariance relies on the following two facts.
The first is that for any tempered (or more generally, unitarizable and generic) irreducible $\pi$
and $m>0$ we have
\[
\Sp(\pi,m)\hookrightarrow\Sp(\pi,m-1)\abs{\cdot}^{-\frac12}\times\pi\abs{\cdot}^{\frac{m-1}2}
\]
where $\Sp(\pi,m)$ is the Langlands quotient of $\pi\abs{\cdot}^{\frac{m-1}2}\times\dots\times\pi\abs{\cdot}^{\frac{1-m}2}$
and $\times$ denotes parabolic induction. (This follows from the fact that the Langlands quotient is obtained as the image of the ``longest''
intertwining operator, which can be factorized as the product of ``shorter'' intertwining operators.)
The second is Bernstein's theorem on $P_n$-invariant distributions which was extended to the archimedean case by Baruch \cite{MR1999922}.


\begin{bibdiv}
\begin{biblist}

\bib{MR3319629}{article}{
      author={Aizenbud, Avraham},
      author={Gourevitch, Dmitry},
      author={Sahi, Siddhartha},
       title={Derivatives for smooth representations of {$GL(n,\Bbb{R})$} and
  {$GL(n,\Bbb{C})$}},
        date={2015},
        ISSN={0021-2172},
     journal={Israel J. Math.},
      volume={206},
      number={1},
       pages={1\ndash 38},
         url={https://doi.org/10.1007/s11856-015-1149-9},
      review={\MR{3319629}},
}

\bib{MR1999922}{article}{
      author={Baruch, Ehud~Moshe},
       title={A proof of {K}irillov's conjecture},
        date={2003},
        ISSN={0003-486X},
     journal={Ann. of Math. (2)},
      volume={158},
      number={1},
       pages={207\ndash 252},
         url={http://dx.doi.org/10.4007/annals.2003.158.207},
      review={\MR{1999922 (2004f:22012)}},
}

\bib{MR0425030}{article}{
      author={Bern{\v{s}}te{\u\i}n, I.~N.},
      author={Zelevinski{\u\i}, A.~V.},
       title={Representations of the group {$GL(n,F),$} where {$F$} is a local
  non-{A}rchimedean field},
        date={1976},
        ISSN={0042-1316},
     journal={Uspehi Mat. Nauk},
      volume={31},
      number={3(189)},
       pages={5\ndash 70},
      review={\MR{0425030 (54 \#12988)}},
}

\bib{MR0579172}{article}{
      author={Bernstein, I.~N.},
      author={Zelevinsky, A.~V.},
       title={Induced representations of reductive {${\germ p}$}-adic groups.
  {I}},
        date={1977},
        ISSN={0012-9593},
     journal={Ann. Sci. \'Ecole Norm. Sup. (4)},
      volume={10},
      number={4},
       pages={441\ndash 472},
      review={\MR{0579172 (58 \#28310)}},
}

\bib{MR748505}{incollection}{
      author={Bernstein, Joseph~N.},
       title={{$P$}-invariant distributions on {${\rm GL}(N)$} and the
  classification of unitary representations of {${\rm GL}(N)$}
  (non-{A}rchimedean case)},
        date={1984},
   booktitle={Lie group representations, {II} ({C}ollege {P}ark, {M}d.,
  1982/1983)},
      series={Lecture Notes in Math.},
      volume={1041},
   publisher={Springer},
     address={Berlin},
       pages={50\ndash 102},
         url={http://dx.doi.org/10.1007/BFb0073145},
      review={\MR{748505 (86b:22028)}},
}

\bib{1812.00047}{misc}{
      author={Beuzart-Plessis, Rapha\"{e}l},
       title={Plancherel formula for ${GL_n(F)\backslash GL_n(E)}$ and
  applications to the {I}chino-{I}keda and formal degree conjectures for
  unitary groups},
        date={2018},
        note={arXiv:1812.00047},
}

\bib{MR946351}{article}{
      author={Cowling, M.},
      author={Haagerup, U.},
      author={Howe, R.},
       title={Almost {$L^2$} matrix coefficients},
        date={1988},
        ISSN={0075-4102},
     journal={J. Reine Angew. Math.},
      volume={387},
       pages={97\ndash 110},
      review={\MR{946351 (89i:22008)}},
}

\bib{MR1957064}{inproceedings}{
      author={Delorme, Patrick},
       title={Harmonic analysis on real reductive symmetric spaces},
        date={2002},
   booktitle={Proceedings of the {I}nternational {C}ongress of
  {M}athematicians, {V}ol. {II} ({B}eijing, 2002)},
   publisher={Higher Ed. Press, Beijing},
       pages={545\ndash 554},
      review={\MR{1957064}},
}

\bib{MR3770165}{article}{
      author={Delorme, Patrick},
       title={Neighborhoods at infinity and the {P}lancherel formula for a
  reductive {$p$}-adic symmetric space},
        date={2018},
        ISSN={0025-5831},
     journal={Math. Ann.},
      volume={370},
      number={3-4},
       pages={1177\ndash 1229},
         url={https://doi.org/10.1007/s00208-017-1554-y},
      review={\MR{3770165}},
}

\bib{MR1100992}{article}{
      author={du~Cloux, Fokko},
       title={Sur les repr\'esentations diff\'erentiables des groupes de {L}ie
  alg\'ebriques},
        date={1991},
        ISSN={0012-9593},
     journal={Ann. Sci. \'Ecole Norm. Sup. (4)},
      volume={24},
      number={3},
       pages={257\ndash 318},
         url={http://www.numdam.org/item?id=ASENS_1991_4_24_3_257_0},
      review={\MR{1100992 (92j:22026)}},
}

\bib{MR3705224}{article}{
      author={Gomez, Raul},
      author={Gourevitch, Dmitry},
      author={Sahi, Siddhartha},
       title={Generalized and degenerate {W}hittaker models},
        date={2017},
        ISSN={0010-437X},
     journal={Compos. Math.},
      volume={153},
      number={2},
       pages={223\ndash 256},
         url={https://doi.org/10.1112/S0010437X16007788},
      review={\MR{3705224}},
}

\bib{MR1078382}{article}{
      author={Heumos, Michael~J.},
      author={Rallis, Stephen},
       title={Symplectic-{W}hittaker models for {${\rm Gl}_n$}},
        date={1990},
        ISSN={0030-8730},
     journal={Pacific J. Math.},
      volume={146},
      number={2},
       pages={247\ndash 279},
         url={http://projecteuclid.org/getRecord?id=euclid.pjm/1102645157},
      review={\MR{1078382 (91k:22036)}},
}

\bib{MR944325}{article}{
      author={Hironaka, Yumiko},
      author={Sat\B{o}, Fumihiro},
       title={Spherical functions and local densities of alternating forms},
        date={1988},
        ISSN={0002-9327},
     journal={Amer. J. Math.},
      volume={110},
      number={3},
       pages={473\ndash 512},
         url={https://doi.org/10.2307/2374620},
      review={\MR{944325}},
}

\bib{MR2058616}{incollection}{
      author={Jacquet, Herv{\'e}},
      author={Lapid, Erez},
      author={Rallis, Stephen},
       title={A spectral identity for skew symmetric matrices},
        date={2004},
   booktitle={Contributions to automorphic forms, geometry, and number theory},
   publisher={Johns Hopkins Univ. Press},
     address={Baltimore, MD},
       pages={421\ndash 455},
      review={\MR{2058616 (2005h:11105)}},
}

\bib{MR1142486}{article}{
      author={Jacquet, Herv{\'e}},
      author={Rallis, Stephen},
       title={Symplectic periods},
        date={1992},
        ISSN={0075-4102},
     journal={J. Reine Angew. Math.},
      volume={423},
       pages={175\ndash 197},
      review={\MR{1142486 (93b:22035)}},
}

\bib{1806.10528}{article}{
      author={Lapid, Erez},
      author={Mao, Zhengyu},
       title={Local {R}ankin--{S}elberg integrals for {S}peh representations},
     journal={Compos. Math.},
      volume={to appear},
        note={arXiv:1806.10528},
}

\bib{MR3267120}{article}{
      author={Lapid, Erez},
      author={Mao, Zhengyu},
       title={A conjecture on {W}hittaker-{F}ourier coefficients of cusp
  forms},
        date={2015},
        ISSN={0022-314X},
     journal={J. Number Theory},
      volume={146},
       pages={448\ndash 505},
         url={http://dx.doi.org/10.1016/j.jnt.2013.10.003},
      review={\MR{3267120}},
}

\bib{MR3431601}{article}{
      author={Lapid, Erez},
      author={Mao, Zhengyu},
       title={Model transition for representations of metaplectic type},
        date={2015},
        ISSN={1073-7928},
     journal={Int. Math. Res. Not. IMRN},
      number={19},
       pages={9486\ndash 9568},
         url={http://dx.doi.org/10.1093/imrn/rnu225},
        note={With an appendix by Marko Tadi{\'c}},
      review={\MR{3431601}},
}

\bib{MR913667}{article}{
      author={M\oe~glin, C.},
      author={Waldspurger, J.-L.},
       title={Mod\`eles de {W}hittaker d\'{e}g\'{e}n\'{e}r\'{e}s pour des
  groupes {$p$}-adiques},
        date={1987},
        ISSN={0025-5874},
     journal={Math. Z.},
      volume={196},
      number={3},
       pages={427\ndash 452},
         url={https://doi.org/10.1007/BF01200363},
      review={\MR{913667}},
}

\bib{MR2248833}{article}{
      author={Offen, Omer},
       title={Residual spectrum of {${\rm GL}_{2n}$} distinguished by the
  symplectic group},
        date={2006},
        ISSN={0012-7094},
     journal={Duke Math. J.},
      volume={134},
      number={2},
       pages={313\ndash 357},
         url={http://dx.doi.org/10.1215/S0012-7094-06-13423-3},
      review={\MR{2248833}},
}

\bib{MR3764130}{article}{
      author={Sakellaridis, Yiannis},
      author={Venkatesh, Akshay},
       title={Periods and harmonic analysis on spherical varieties},
        date={2017},
        ISSN={0303-1179},
     journal={Ast\'erisque},
      number={396},
       pages={viii+360},
      review={\MR{3764130}},
}

\bib{1812.04091}{misc}{
      author={Smith, Jerrod~Manford},
       title={Speh representations are relatively discrete},
        date={2018},
        note={arXiv:1812.04091},
}

\bib{MR1989693}{article}{
      author={Waldspurger, J.-L.},
       title={La formule de {P}lancherel pour les groupes {$p$}-adiques
  (d'apr\`es {H}arish-{C}handra)},
        date={2003},
        ISSN={1474-7480},
     journal={J. Inst. Math. Jussieu},
      volume={2},
      number={2},
       pages={235\ndash 333},
         url={https://doi.org/10.1017/S1474748003000082},
      review={\MR{1989693}},
}

\bib{MR584084}{article}{
      author={Zelevinsky, A.~V.},
       title={Induced representations of reductive {${\germ p}$}-adic groups.
  {II}. {O}n irreducible representations of {${\rm GL}(n)$}},
        date={1980},
        ISSN={0012-9593},
     journal={Ann. Sci. \'Ecole Norm. Sup. (4)},
      volume={13},
      number={2},
       pages={165\ndash 210},
         url={http://www.numdam.org/item?id=ASENS_1980_4_13_2_165_0},
      review={\MR{584084 (83g:22012)}},
}

\end{biblist}
\end{bibdiv}

\def\cprime{$'$}

\end{document}